\RequirePackage[french, english]{babel}
\documentclass[a4paper,11pt]{amsart}

\usepackage{amssymb}
\usepackage{amsmath}
\usepackage{amsfonts}

\usepackage{pstricks}
\usepackage{amsthm}

\usepackage{verbatim}

\usepackage{hyperref}
\usepackage{amscd}
\usepackage[all]{xy}

\DeclareSymbolFont{rsfs}{U}{rsfs}{m}{n}
\DeclareSymbolFontAlphabet{\mathcal}{rsfs}

\DeclareFontEncoding{OT2}{}{} 
\DeclareTextFontCommand{\textcyr}{\fontencoding{OT2}
    \fontfamily{wncyr}\fontseries{m}\fontshape{n}\selectfont}

\theoremstyle{plain}

\newtheorem{theo}{Th\'eor\`eme}[section]
\newtheorem*{thm}{Th\'eor\`eme}

\newtheorem{defi}[theo]{D\'efinition}
\newtheorem{prop}[theo]{Proposition}

\newtheorem{cor}[theo]{Corollaire}
\newtheorem{lem}[theo]{Lemme}

\newenvironment{exs}{\noindent{\textbf{Exemples :}}}{}

\theoremstyle{definition}

\newtheorem{rem}[theo]{Remarque}

\def \Romannumeral #1 {\expandafter\uppercase\expandafter {\romannumeral #1} }

\def \Br {{\rm{Br}}}

\def \tor {{\rm{tor}}}

\def \Aut{{\rm Aut \,}}

\def\ov{\overline}

\def\Gm{{\mathbf{G}_m}}

\newcommand{\xto}[1]{\xrightarrow{#1}}

\def\uu{^\mathrm{u}}

\def\red{^\mathrm{red}}
\def\tor{^{\mathrm{tor}}}
\def\sc{^{\mathrm{sc}}}
\renewcommand{\ss}{^{\mathrm{ss}}}

\def\lin{^{\mathrm{lin}}}
\def\li{_{\mathrm{lin}}}
\def\SA{\textup{SA}}

\newcommand{\ensemble}[1]{\ensuremath{\mathbf #1} \xspace}
  \renewcommand{\H}{\ensemble H}

\def\R{{\mathbf{R}}}

\newcommand{\abvar}{{}^{\textup{ab}}}

\def\Br{{\rm Br}}
\def\Pic{{\rm Pic}}
\def\UPic{{\rm UPic}}
\def\Div{{\rm Div}}
\def\SAut{{\rm SAut}}
\def\Aut{{\rm Aut}}

\title[Groupe de Brauer alg\'ebrique d'un torseur]
{        
Une formule pour le groupe de Brauer alg\'ebrique d'un torseur \\
}

\author{Cyril Demarche}
\address{Cyril Demarche
\newline Universit\'e Pierre et Marie Curie (Paris 6) \newline
Institut de Math\'ematiques de Jussieu \newline
4 place Jussieu, 75252 Paris Cedex 05 \newline
France
}

\email{demarche@math.jussieu.fr}

\begin{document}


\date{23 novembre 2010}


\maketitle

\selectlanguage{english}

\begin{abstract}
For a homogeneous space $X$ of a connected algebraic group $G$ (with
 connected stabilizers) over a
 field $k$ of characteristic zero, we construct a canonical complex of Galois modules of length $3$ and a canonical
isomorphism between an hypercohomology group of this complex and an
explicit subgroup of the Brauer group of $X$. This result is obtained
as a consequence of a formula describing the ``algebraic Brauer group of a
torsor'', and it generalizes recent results by Borovoi and van Hamel,
by considering non-linear groups $G$ and by taking into account some
transcendental elements in the Brauer group of $X$. 
\end{abstract}

\selectlanguage{french}

\begin{abstract}
\'Etant donn\'e un espace homog\`ene $X$ d'un groupe alg\'ebrique
connexe $G$ (\`a stabilisateurs connexes) sur un corps $k$ de
 caract\'eristique nulle, on construit un complexe de modules galoisiens
 de longueur $3$ et un isomorphisme canonique entre un groupe d'hypercohomologie de ce
complexe et un sous-groupe explicite du groupe de Brauer de $X$. Ce
r\'esultat est une cons\'equence d'une formule d\'ecrivant le ``groupe
de Brauer alg\'ebrique d'un torseur'', et il g\'en\'eralise des r\'esultats r\'ecents de Borovoi et van
Hamel, en consid\'erant des groupes $G$ non lin\'eaires, et en prenant
en compte certains \'el\'ements transcendants du groupe de
Brauer de $X$.
\end{abstract}

\def\pn{\par\noindent}

\bigskip

\setcounter{section}{-1}

\section{Introduction}

L'objectif principal de cet article est de comprendre le groupe de Brauer (ou au moins une partie de celui-ci) d'un espace homog\`ene d'un groupe alg\'ebrique connexe \`a stabilisateurs connexes et d'en obtenir une description explicite. L'un des int\'er\^ets de la formule d\'emontr\'ee ici r\'eside dans le fait qu'elle prenne en compte une partie des \'el\'ements transcendants du groupe de Brauer, et pas seulement le groupe de Brauer alg\'ebrique.

Pr\'ecisons quelques notations : $k$ est un corps de caract\'eristique
nulle, $\ov{k}$ une cl\^oture alg\'ebrique de $k$, $\Gamma_k$ le groupe
de Galois absolu de $k$, $G$ un $k$-groupe alg\'ebrique connexe (pas forc\'ement lin\'eaire) et $X$ un espace homog\`ene de $G$, \`a stabilisateurs g\'eom\'etriques connexes. On note $\ov{H}$ le stabilisateur (d\'efini sur $\ov{k}$ a priori) d'un point ferm\'e $\ov{x} \in X(\ov{k})$. Remarquons qu'un tel point $\ov{x}$ d\'efinit un morphisme de $\ov{k}$-vari\'et\'es $\pi : \ov{G} \to \ov{X}$. On d\'efinit alors $\Br_1(X,G)$ comme le noyau de l'application compos\'ee :
$$\Br(X) \to \Br(\ov{X}) \xrightarrow{\pi^*} \Br(\ov{G}) \, .$$
Il est clair que le groupe $\Br_1(X,G)$ contient le groupe de Brauer alg\'ebrique $\Br_1(X) := \textup{Ker}(\Br(X) \to \Br(\ov{X}))$, et il peut \^etre strictement plus gros. On notera aussi $\Br_a(X,G) := \Br_1(X,G) / \Br(k)$.

Par le th\'eor\`eme de Chevalley (voir \cite{Che} et \cite{Ros}, th\'eor\`eme 16; voir \'egalement \cite{Con} pour une preuve ``moderne''), le groupe $G$ est extension d'une vari\'et\'e ab\'elienne par un groupe lin\'eaire connexe
$$1 \to G\lin \to G \to G\abvar \to 1 \, .$$ 
On note $G\red$ le quotient de $G\lin$ par son radical unipotent, $G\ss$ le groupe d\'eriv\'e de $G\red$ et $G\sc$ le rev\^etement simplement connexe de $G\ss$. Soit $T_G$ (resp. $T_{G\sc}$) un tore maximal de $G\red$ (resp. $G\sc$). Un r\'esultat g\'en\'eral sur les espaces homog\`enes (voir \cite{BCTS}, proposition 3.1) assure que, quitte \`a remplacer $G$ par un groupe connexe $G'$, on peut supposer $\ov{H}$ lin\'eaire. Notons que les centres $Z_{\ov{H}\red}$ et $Z_{\ov{H}\sc}$ des groupes $\ov{H}\red$ et $\ov{H}\sc$ sont canoniquement d\'efinis sur $k$. Le module des caract\`eres d'un groupe alg\'ebrique $F/k$ est not\'e $\widehat{F} := \textup{Hom}_{\ov{k}-\textup{gr}}(\ov{F}, \Gm_{\ov{k}})$.
On est en mesure d'\'enoncer l'un des r\'esultats principaux :
\begin{thm}
Soit $X$ un $k$-espace homog\`ene d'un groupe alg\'ebrique connexe $G / k$, \`a stabilisateur g\'eom\'etrique lin\'eaire connexe $\ov{H}$. 
\begin{itemize}
  \item On suppose $\Pic(\ov{G}) = 0$. D\'efinissons le complexe naturel de modules galoisiens
$$\widehat{\ov{C}}_X := \left[ \widehat{G} \to \widehat{Z_{\ov{H}\red}} \to \widehat{Z_{\ov{H}\sc}} \right] \, ,$$
avec $\widehat{G}$ en degr\'e $0$. On a alors une suite exacte canonique
$$0 \to \Br_a(X,G) \to \H^2(k, \widehat{\ov{C}}_X) \to \textup{Ker}(H^3(k, \Gm) \to H^3(X, \Gm)) \, ,$$
qui induit un isomorphisme $\Br_a(X,G) \cong \H^2(k, \widehat{\ov{C}}_X)$ si $X(k) \neq \emptyset$ ou $H^3(k, \Gm) = 0$.
\item On suppose $X(k) \neq \emptyset$ (on ne suppose plus $\Pic(\ov{G}) = 0$). D\'efinissons le complexe
$$\widehat{C}'_X := \left[\widehat{T_G} \to \Pic(\ov{G}\abvar) \oplus \widehat{T_{G\sc}} \oplus \widehat{T_H} \to \widehat{T_{H\sc}} \right] \, ,$$
avec $\widehat{T_G}$ en degr\'e $0$. On a alors un isomorphisme canonique 
$$\Br_a(X,G) \cong \H^2(k, \widehat{C}'_X) \, .$$
\end{itemize}
\end{thm}

Remarquons au passage que l'hypoth\`ese $H^3(k, \Gm) = 0$ dans la premi\`ere partie est v\'erifi\'ee par exemple si $k$ est un corps de nombres, un corps $p$-adique, ou un corps de fonctions d'une courbe sur un corps de nombres ou sur un corps $p$-adique. On voit aussi qu'il suffit que le morphisme $H^3(k, \Gm) \to H^3(X, \Gm)$ soit injectif pour avoir un isomorphisme, ce qui est par exemple le cas lorsque $X$ a un z\'ero-cycle de degr\'e $1$.

Ce r\'esultat s'inscrit \`a la suite d'un ensemble de travaux sur le groupe de Brauer des groupes alg\'ebriques et de leurs espaces homog\`enes. En 1981, Sansuc d\'emontre (voir \cite{San}, lemme 6.9) une formule d\'ecrivant le groupe $\Br_a(X)$ quand $X$ est un tore ou un groupe semi-simple. Dans \cite{BvH}, Borovoi et van Hamel introduisent le complexe $\UPic(\ov{X})$ et g\'en\'eralisent les r\'esultats de Sansuc dans le cas o\`u $X$ est un groupe lin\'eaire connexe (voir \cite{BvH}, corollaire 7), reformulant par l\`a m\^eme un r\'esultat de Kottwitz (voir \cite{Kot}, 2.4). Puis les m\^emes auteurs calculent le groupe $\Br_a(X)$ pour un espace homog\`ene d'un groupe lin\'eaire connexe : voir \cite{BvH2}, corollaire 3.2 et \cite{BvHprep}, th\'eor\`eme 7.2. Mentionnons \'egalement que Harari et Szamuely construisent dans \cite{HSz2}, section 4, un morphisme canonique $\iota : \H^1(k, M^*) \to \Br_a(X)$ compatible avec l'obstruction de Brauer-Manin, o\`u $X$ est un torseur sous une vari\'et\'e semi-ab\'elienne $S$ et $M^*$ est le $1$-motif dual de $S$. Enfin, on peut citer un r\'esultat concernant le groupe de Brauer non ramifi\'e d'un espace homog\`ene : Colliot-Th\'el\`ene et Kunyavski\u{\i} ont obtenu une formule d\'ecrivant le groupe de Brauer alg\'ebrique d'une compactification lisse d'un espace homog\`ene \`a stabilisateurs connexes (voir \cite{CTK}, th\'eor\`eme A).

Le th\'eor\`eme principal de ce texte est donc \`a la fois une unification et une g\'en\'eralisation de tous ces r\'esultats (hormis celui concernant le groupe de Brauer non ramifi\'e). L'int\'er\^et principal de cette g\'en\'eralisation est la prise en compte de certains \'el\'ements transcendants du groupe de Brauer, contrairement aux r\'esultats mentionn\'es qui se limitaient \`a $\Br_a(X)$. Outre l'int\'er\^et th\'eorique de d\'ecrire l'ensemble du groupe $\Br(X)$, l'une des motivations pour l'introduction et l'\'etude du groupe $\Br_a(X,G)$ r\'eside dans les r\'esultats r\'ecents obtenus par Borovoi et l'auteur dans \cite{BD} \`a propos du d\'efaut d'approximation forte dans les espaces homog\`enes sur les corps de nombres (voir notamment le th\'eor\`eme 1.4 de \cite{BD}). Le groupe $\Br_a(X,G)$ y appara\^it en effet de fa\c con naturelle comme le sous-groupe minimal de $\Br(X)$ permettant de d\'ecrire l'adh\'erence des points rationnels dans l'ensemble des points ad\'eliques de $X$, \`a l'aide de l'obstruction de Brauer-Manin. En effet, si le groupe de Brauer alg\'ebrique est suffisant pour l'obstruction de Brauer-Manin au principe de Hasse et \`a l'approximation faible sur $X$, ce n'est pas le cas pour l'obstruction \`a l'approximation forte et au principe de Hasse pour les points entiers, o\`u une obstruction transcendante est en g\'en\'eral n\'ecessaire (voir \cite{BD} contre-exemple 1.6 et \cite{CTX} remarque 2.11). Ainsi le th\'eor\`eme \'enonc\'e dans cette introduction, en regard du r\'esultat principal de \cite{BD}, donne-t-il une formule explicite et calculable, en termes de l'hypercohomologie d'un complexe de modules galoisiens, pour le d\'efaut d'approximation forte dans un espace homog\`ene sur un corps de nombres. 

Mentionnons \'egalement que l'on obtient au passage dans ce texte des
r\'esultats g\'en\'eraux concernant le ``groupe de Brauer alg\'ebrique
des torseurs'' : si $\pi : Y \xrightarrow{H} X$ est un torseur sous un
groupe lin\'eaire connexe $H$, on donne une formule (voir section
\ref{section torseurs}, th\'eor\`eme \ref{theo torseurs k} et ses
corollaires) d\'ecrivant le groupe $\Br_1(X,Y) := \textup{Ker}(\Br(X)
\xrightarrow{\pi^*} \Br(Y) \to \Br(\ov{Y}))$. Cette formule est une
g\'en\'eralisation naturelle des r\'esultats classiques de Sansuc sur le
groupe de Brauer des torseurs.

\subsection*{Remerciements} 
Je remercie chaleureusement Mikhail Borovoi, Jean-Louis
Colliot-Th\'el\`ene et David Harari pour leurs pr\'ecieux commentaires
sur cet article.

\section{Pr\'eliminaires sur les groupes de Picard et de Brauer des
  torseurs}

Dans tout ce texte, sauf mention explicite du contraire, on entend par ``cat\'egorie d\'eriv\'ee'' la cat\'egorie d\'eriv\'ee associ\'ee \`a la cat\'egorie des complexes born\'es de modules galoisiens sur $k$. Dans tout le texte, \'etant donn\'e un sch\'ema $X$, on consid\`ere toujours le site \'etale sur $X$, les faisceaux consid\'er\'es sont des faisceaux \'etales sur $X$ et la cohomologie consid\'er\'ee est la cohomologie \'etale de $X$.

\subsection{Un complexe associ\'e \`a un morphisme}
\label{subsection morphisme}
On construit ici un complexe de modules galoisiens, de longueur $3$,
associ\'e \`a un morphisme de $k$-vari\'et\'es alg\'ebriques $\pi : Y \to X$. Ce complexe est crucial en vue
du calcul du groupe de Brauer dans la section \ref{section Brauer torseurs descente}.

Soit $\pi : Y \to X$ un morphisme de $k$-vari\'et\'es alg\'ebriques lisses et g\'eom\'etriquement int\`egres.
On consid\`ere la suite exacte habituelle de faisceaux (\'etales) sur $Y$ (voir \cite{GBr}, II.1, suite exacte (2)) : 
\begin{equation}
\label{sec Y}
0 \to \Gm_Y \to \mathcal{K}_Y^* \to \mathcal{D}iv_Y \to 0 \, .
\end{equation}
On applique le foncteur $\pi_*$ \`a cette suite exacte. On obtient la
suite exacte suivante :
$$0 \to {\pi}_* \Gm_Y \to {\pi}_* \mathcal{K}_Y^* \to \pi_* \mathcal{D}iv_Y \to
R^1 \pi_* \Gm_Y \to R^1 \pi_* \mathcal{K}_Y^* \, .$$
Or le th\'eor\`eme de Hilbert 90 assure que $R^1 {\pi}_* \mathcal{K}_Y^* = 0$ (voir \cite{GBr}, II, lemme 1.6), et donc on a une suite exacte
\begin{equation}
\label{sel base}
0 \to {\pi}_* \Gm_Y \to {\pi}_* \mathcal{K}_Y^* \to {\pi}_* \mathcal{D}iv_Y \to
R^1 \pi_* \Gm_Y \to 0 \, .
\end{equation}
Notons $\mathcal{Q}$ le conoyau du morphisme $\pi_* \Gm_Y \to \pi_*
\mathcal{K}_Y^*$, et appliquons le foncteur ${p_X}_*$ \`a cette suite exacte
(o\`u $p_X : X \to \textup{Spec}(k)$ est le morphisme structural de $X$). On obtient les suites exactes suivantes (en
notant que $p_Y = p_X \circ \pi$) :
$$0 \to {p_Y}_* \Gm_Y \to {p_Y}_* \mathcal{K}_Y^* \to {p_X}_*
\mathcal{Q} \to R^1 {p_X}_* \pi_* \Gm_Y \to  R^1 {p_X}_* \pi_*
\mathcal{K}_Y^*$$
$$R^1 {p_X}_* \pi_* \mathcal{K}_Y^* \to R^1 {p_X}_* \mathcal{Q} \to R^2
{p_X}_* \pi_* \Gm_Y \to R^2 {p_X}_* \pi_* \mathcal{K}_Y^*$$
$$0 \to {p_X}_* \mathcal{Q} \to {p_Y}_* \mathcal{D}iv_Y \to {p_X}_*
R^1 \pi_* \Gm_Y \to R^1 {p_X}_* \mathcal{Q} \, .$$
Or le faisceau $R^1 {p_X}_* \pi_* \mathcal{K}_Y^*$ est nul car il s'injecte dans $R^1 {p_Y}_*
\mathcal{K}_Y^* = 0$ par un argument de suite spectrale, donc les
suites pr\'ec\'edentes se r\'e\'ecrivent ainsi :
\begin{equation}
\label{sel1}
0 \to \overline{k}[Y]^* \to \overline{k}(Y)^* \to {p_X}_*
\mathcal{Q} \to R^1 {p_X}_* \pi_* \Gm_Y \to 0
\end{equation}
\begin{equation}
\label{sel2}
0 \to R^1 {p_X}_* \mathcal{Q} \to R^2 {p_X}_* \pi_* \Gm_Y \to R^2
{p_X}_* \pi_* \mathcal{K}_Y^*
\end{equation}
\begin{equation}
\label{sel3}
0 \to {p_X}_* \mathcal{Q} \to \Div(\overline{Y}) \to {p_X}_* R^1 \pi_*
\Gm_Y \to R^1 {p_X}_* \mathcal{Q} \, .
\end{equation}
Ces suites exactes permettent de d\'efinir le complexe suivant (o\`u $\overline{k}(Y)^*/\overline{k}^*$ est en degr\'e $0$) 
$$\textup{UPic}(\pi : Y \to X) := [\overline{k}(Y)^*/\overline{k}^* \xrightarrow{\Delta}
\Div(\overline{Y}) \xrightarrow{\partial} H^0(\overline{X}, R^1 \pi_*
\Gm_Y)] \, ,$$
o\`u le morphisme $\Delta$ est le morphisme $f \mapsto
\textup{div}(f)$, \'egal \`a la compos\'ee des morphismes suivants
apparaissant dans les suites exactes (\ref{sel1}) et (\ref{sel3}) : $\overline{k}(Y)^* \to {p_X}_* \mathcal{Q} \to \Div(\overline{Y})$,
et le morphisme $\partial$ apparait dans la suite exacte
(\ref{sel3}). Notons enfin que le groupe $H^0(\overline{X}, R^1 \pi_*
\Gm_Y)$ est appel\'e le groupe de Picard relatif de $\ov{Y}$ sur
$\ov{X}$, not\'e $\Pic'(\ov{Y}/\ov{X})$ (voir \cite{BLR}, chapitre 8 :
le groupe $\Pic'(\ov{Y}/\ov{X})$ est le groupe des sections globales sur
$\ov{X}$ du foncteur de Picard relatif $\Pic_{\ov{Y} / \ov{X}}$).

Le lemme suivant est alors \'evident (on rappelle que $\UPic(\ov{Y})$
est repr\'esent\'e par $[\ov{k}(Y)^*/\ov{k} \to \Div(\ov{Y})]$ : voir
\cite{BvH}, corollaire 2.5) :
\begin{lem}
On a un triangle exact canonique dans la cat\'egorie d\'eriv\'ee des
 modules galoisiens :
$$\Pic'(\ov{Y}/\ov{X})[-2] \to \UPic(\pi : Y \to X) \to \UPic(\ov{Y})
\to \Pic'(\ov{Y}/\ov{X})[-1] \, .$$
\end{lem}

\subsection{Application au cas des torseurs}
\label{section donnee descente}
Par convention, et sauf mention explicite du contraire, les torseurs consid\'er\'es sont des \emph{torseurs \`a droite}.

Soit $k$ un corps de caract\'eristique nulle, $X$ une $k$-vari\'et\'e lisse et g\'eom\'etriquement int\`egre. Soit $\ov{H}$ un $\ov{k}$-groupe alg\'ebrique lin\'eaire connexe et $\ov{\pi} : \ov{Y}
\xrightarrow{\ov{H}} \ov{X}$ un $\ov{X}$-torseur sous $\ov{H}$. 

On rappelle ici la d\'efinition d'une donn\'ee de recollement sur un tel
torseur (voir par exemple \cite{HSk}, d\'efinition 2.1, o\`u cette notion est appel\'ee
``donn\'ee de descente''). On renvoie \`a \cite{HSk} pour les notations :
\begin{defi}
\label{defi rec}
Une donn\'ee de recollement sur $\ov{Y}/X$ est la donn\'ee d'un
 sous-groupe topologique $E$ de $\SAut(\ov{Y}/X)$ s'int\'egrant dans un
 diagramme commutatif exact de groupes topologiques
\begin{displaymath}
\xymatrix{
1 \ar[r] & \ov{H}(\ov{k}) \ar[r] \ar[d]^i & E \ar[r] \ar[d] & \Gamma_k
 \ar[r] \ar[d] & 1 \\
1 \ar[r] & \Aut(\ov{Y}/\ov{X}) \ar[r] & \SAut(\ov{Y}/X) \ar[r] &
 \Gamma_k &
}
\end{displaymath}
o\`u la ligne sup\'erieure est localement scind\'ee et le morphisme $i$
 est le morphisme naturel.
\end{defi}

\`A partir de maintenant, \emph{on suppose que $\ov{Y} / X$ est muni d'une donn\'ee de recollement.}


Dans cette section, on associe \`a la donn\'ee du torseur $\ov{Y} / X$, muni de sa donn\'ee de recollement, un complexe de modules galoisiens not\'e $C_{\ov{Y}/X}$, que l'on relie \`a un complexe de la forme $\UPic(\pi : Z \to X)$, o\`u $Z/X$ est un torseur (d\'efini sur $k$) associ\'e \`a $\ov{Y}/X$.

Notons $\ov{H}\uu$ le radical unipotent de $\ov{H}$, et $\ov{H}\red := \ov{H} / \ov{H}\uu$. D\'efinissons le groupe $\ov{H'} := \ov{H}\red / Z_{\ov{H}\red}$ et les quotients $\ov{Z'} := \ov{Y} / \ov{H}\uu$ et $\ov{Z} := \ov{Z'} / Z_{\ov{H}\red}$, de sorte que l'on ait un diagramme commutatif de torseurs sur $\ov{X}$ :
\begin{displaymath}
\xymatrix{
\ov{Y} \ar[r]^{\ov{H}\uu} \ar[ddr]_{\ov{H}} & \ov{Z'} \ar[dd]^{\ov{H}\red} \ar[rd]^{Z_{\ov{H}\red}} & \\
& & \ov{Z} \ar[ld]^{\ov{H'}} \\
& \ov{X} & \, .
}
\end{displaymath}

Remarquons au passage que les quotients $\ov{Z'}$ et $\ov{Z}$ sont repr\'esentables par des $\ov{k}$-vari\'et\'es (voir par exemple \cite{Mil}, th\'eor\`eme III.4.3.(a)).

Suivant \cite{HSk}, proposition 2.2, on peut associer \`a une telle donn\'ee de recollement un $k$-lien $L_{\ov{Y}/X}$ sur le groupe $\ov{H}$, ainsi qu'une classe $\eta_{\ov{Y}/X} \in H^2(k,L_{\ov{Y}/X})$, qui est l'obstruction \`a descendre le torseur $\ov{Y} \to \ov{X}$ en un torseur sur $X$ : cette classe est neutre si et seulement si le torseur $\ov{Y} / \ov{X}$ descend sur $X$.

Le sous-groupe $\ov{H}\uu$ de $\ov{H}$ est invariant par les semi-automorphismes de $\ov{H}$, donc le lien $L_{\ov{Y}/X}$ induit un lien
$L_{\ov{Z'}/X}$ sur $\ov{H}\red$. De m\^eme, $L_{\ov{Z'}/X}$ induit un lien $L_{\ov{Z}/X}$ sur $\ov{H'}$. Or $\ov{H'}$ est r\'eductif et de centre trivial, donc par la proposition 1.1 de \cite{Dou} ou la proposition 3.1 de \cite{BorDuke}, l'ensemble $H^2(k,L_{\ov{Z}/X})$ est r\'eduit \`a un seul \'el\'ement, qui est une classe neutre. Donc l'image $\eta'_{\ov{Y}/X} \in H^2(k, L_{\ov{Z}/X})$ de la classe $\eta_{\ov{Y}/X} \in H^2(k, L_{\ov{Y}/X})$ est neutre.

Or par fonctorialit\'e, la classe $\eta'_{\ov{Y}/X}$ co\"incide avec la
classe associ\'ee \`a la donn\'ee de recollement induite sur $\ov{Z} \to \ov{X}$. Par cons\'equent, puisque $\eta'_{\ov{Y}/X}$ est neutre, il existe une $k$-forme $H'$ de $\ov{H'}$ et un torseur $p : Z \xrightarrow{H'} X$ qui devient isomorphe \`a $\ov{p} : \ov{Z} \xrightarrow{\ov{H'}} \ov{X}$ quand on \'etend les scalaires \`a $\ov{k}$.

Gr\^ace au lemme 5.2.(ii) de \cite{BvH}, on dispose d'un morphisme canonique $\varphi : \UPic(\ov{Z}) \to \UPic(\ov{H'})$. Ce morphisme est repr\'esentable par le morphisme de complexes horizontaux :
\begin{displaymath}
\xymatrix{
\ov{k}(Z)^* / \ov{k}^* \ar[r] \ar[d] & \textup{Div}(\ov{Z}) \ar[d]^{\varphi'} \\
0 \ar[r] & \Pic(\ov{H'}) \, ,
}
\end{displaymath}
via le quasi-isomorphisme canonique $\UPic(\ov{H'}) \to \Pic(\ov{H'})[-1]$, 
o\`u $\varphi'$ est la compos\'ee
$$\textup{Div}(\ov{Z}) \to \Pic(\ov{Z}) \xrightarrow{\varphi} \Pic(\ov{H'}) \, .$$

D\'efinissons l'objet $C_{\ov{Y}/X}$ comme le complexe de modules galoisiens
$$C_{\ov{Y}/X} := [\ov{k}(Z)^* / \ov{k}^* \to \textup{Div}(\ov{Z})
\xrightarrow{\varphi'} \Pic(\ov{H'})] \, ,$$
o\`u $\ov{k}(Z)^* / \ov{k}^*$ est en degr\'e $0$ et $\Pic(\ov{H'})$ en degr\'e $2$.
Par construction, $C_{\ov{Y}/X}[1]$ est un c\^one du morphisme $\varphi : \UPic(\ov{Z}) \to \UPic(\ov{H'})$ dans la cat\'egorie d\'eriv\'ee : on a un triangle exact
$$ C_{\ov{Y}/X} \to \UPic(\ov{Z}) \to \UPic(\ov{H'}) \to C_{\ov{Y}/X}[1] \, .$$

On souhaite maintenant d\'efinir un morphisme naturel $C_{\ov{Y}/X} \to \UPic(\pi : Z \to X)$.

\subsubsection{Construction g\'en\'erale}
\label{section W/X}
Soit $f : V \xrightarrow{G} X$ un $X$-torseur sous un $k$-groupe lin\'eaire connexe $G$. Construisons un morphisme canonique 
$$s_{V/X} : \Pic(G) \to \Pic'(V/X) \, .$$
Soit une classe $p \in \Pic(G)$ repr\'esent\'ee par un torseur $f : P \to G$ sous $\Gm$. Par d\'efinition, on a un isomorphisme naturel de la forme $\phi : V \times G \to V
\times_X V$. On d\'efinit alors $W$ comme le tir\'e en arri\`ere
du torseur $f_V : V \times P \to V \times G$ sous $\Gm$ (d\'eduit de $P \to G$
par changement de base par la projection $V \times G \to G$) par le
morphisme $\phi^{-1} : V \times_X V \to V \times G$. On obtient ainsi
un torseur $W \to V \times_X V$ sous $\Gm$, qui d\'efinit bien une
classe dans $\Pic'(V/X)$ (voir \cite{BLR} p.202 pour
une description explicite des \'el\'ements de $\Pic'(V/X)$). On note alors par d\'efinition $s_{V/X}(p)$ la
classe de $W \to V \times_X V$ dans $\Pic'(V/X)$, et on v\'erifie que
cette classe ne d\'epend pas du repr\'esentant $P$ de $p$ choisi.

Le morphisme $s_{V/X}$ induit ainsi un diagramme canonique 
\begin{equation}
\label{diag UPic Pic'}
\xymatrix{
\Pic(V) \ar[r]^{\varphi} \ar[d]^= & \Pic(G) \ar[d]^{s_{V/X}} \\
\Pic(V) \ar[r]^{\partial} & \Pic'(V/X) \, .
}
\end{equation}

\begin{lem}
\label{lem diag UPic comm}
Le diagramme (\ref{diag UPic Pic'}) est commutatif.
\end{lem}

\begin{proof}
On commence par rappeler la d\'efinition de $\varphi$. Si $m : V \times G \to V$ d\'esigne l'action de $G$, alors $\varphi$ est la compos\'ee des morphismes suivants :
$$\Pic(V) \xrightarrow{m^*} \Pic(V \times G) \xleftarrow{\cong}
\Pic(V) \oplus \Pic(G) \rightarrow \Pic(G) \, ,$$
(voir \cite{San}, lemme 6.6 pour l'isomorphisme central).
Soit $f : Z \to V$ un torseur sous $\Gm$ de classe $p \in
\Pic(V)$. Par d\'efinition de $\varphi$, il existe un torseur $W \to
G$ sous $\Gm$ (de classe $[W] = \varphi(p)$ dans $\Pic(G)$) et un
isomorphisme de $V \times G$-torseurs sous $\Gm$ entre $Z' := m^* Z$
et le quotient de $Z \times W$ par l'action diagonale de $\Gm$. Par
cons\'equent, $s_{V/X}(\varphi(p)) \in \Pic'(V/X)$ est repr\'esent\'e
par le torseur $U \xrightarrow{\Gm} V \times_X V$ d\'eduit de $V
\times W \to V \times G$ via l'isomorphisme $V \times G \cong V
\times_X V$. Quant \`a la classe $\partial(p)$, elle est repr\'esent\'ee
par le torseur $Z \xrightarrow{\Gm} V$.

On cherche donc \`a montrer que les classes des torseurs $Z
\xrightarrow{\Gm} V$ et $U \xrightarrow{\Gm} V \times_X V$
co\"incident dans $\Pic'(V/X)$. Pour cela, il suffit de montrer que
dans les diagrammes de carr\'es cart\'esiens suivants :
\begin{displaymath}
\xymatrix{
Z'' \ar[r] \ar[d] & Z' \ar[r] \ar[d] & Z \ar[d] & U' \ar[r] \ar[d] & U
\ar[d] & \\
Z \times_X V \ar[r] \ar[d] & V \times_X V \ar[r] \ar[d] & V \ar[d] &
Z \times_X V \ar[r] \ar[d] & V \times_X V \ar[d] \ar[r] & V \ar[d] \\
Z \ar[r] & V \ar[r] & X & Z \ar[r] & V \ar[r] & X
}
\end{displaymath}
les $Z \times_X V$-torseurs $Z''$ et $U'$ sont isomorphes. Il suffit
donc de montrer que $[Z''] = [U'] \in \Pic(Z \times_X V)$. Pour cela,
consid\'erons le diagramme commutatif suivant :
\begin{displaymath}
\xymatrix{
\Pic(V \times_X V) \ar[r]^{\cong} \ar[d] & \Pic(V \times G) \ar[d] &
\Pic(V) \oplus \Pic(G) \ar[l]_{\cong} \ar@<4ex>[d]^= \ar@<-4ex>[d] \\
\Pic(Z \times_X V) \ar[r]^{\cong} & \Pic(Z \times G) & \Pic(Z) \oplus
\Pic(G) \ar[l]_{\cong}  \, .
}
\end{displaymath}
Par construction, la classe $[Z'] \in \Pic(V \times G)$ correspond au
couple $([Z], [W]) \in \Pic(V) \oplus \Pic(G)$. Donc, par
fonctorialit\'e, son image $[Z'']$ dans $\Pic(Z \times G)$ correspond au
couple $(0, [W]) \in \Pic(Z) \oplus \Pic(G)$ (on rappelle que le
tir\'e en arri\`ere du torseur $Z \to V$ par lui-m\^eme est trivial
comme $Z$-torseur).

La classe $[U] \in \Pic(V \times_X V)$ correspond par construction \`a la classe $[V \times W] \in \Pic(V
\times G)$, qui elle-m\^eme correspond au couple $(0, [W]) \in \Pic(V)
\oplus \Pic(G)$. On en d\'eduit donc que la classe $[U'] \in \Pic(Z
\times_X V)$ correspond au couple $(0,[W]) \in \Pic(Z) \oplus
\Pic(G)$.

Par cons\'equent, on a montr\'e que $[Z''] = [U'] \in \Pic(Z
\times_X V)$, i.e. $s_{V/X}(\varphi(p)) = \partial(p)$.
\end{proof}

On d\'eduit imm\'ediatemment du lemme \ref{lem diag UPic comm} la construction suivante :
\begin{lem}
Soit $f : V \xrightarrow{G} X$ un torseur sous un groupe lin\'eaire connexe $G$ tel que $G\tor = 0$.
Alors le morphisme $s_{\ov{V}/\ov{X}} : \Pic(\ov{G}) \to \Pic'(\ov{V}/\ov{X})$ induit un
 morphisme naturel :
$$C_{\ov{V}/X} \to \UPic(f : V \to X) \, .$$
\end{lem}

\begin{prop}
\label{prop qis caracteres}
Soit $f : V \xrightarrow{G} X$ un torseur sous un groupe lin\'eaire connexe $G$.
Supposons que $G\tor = 0$ (i.e. $\widehat{G} = 0$). Alors le morphisme naturel
$$\Gm_X \xrightarrow{\cong} f_* \Gm_V$$
est un isomorphisme, et le morphisme canonique 
$$C_{\ov{V}/X} \to \UPic(f : V \to X)$$
est un quasi-isomorphisme.
\end{prop}

\begin{proof}
Pour le premier point, il suffit d'\'etendre la proposition 14.2 de \cite{CTS} (lemme de Rosenlicht global) aux torseurs sous des groupes connexes quelconques :
\begin{prop}
\label{prop Rosenlicht}
Soit $X$ un sch\'ema r\'eduit et $G / X$ un sch\'ema en groupes plat
localement de pr\'esentation finie, \`a fibres maximales lisses
connexes. Soit $f : V \xrightarrow{G} X$ un torseur sous $G$. On a
alors une suite exacte de faisceaux \'etales sur $X$ :
$$1 \to \Gm_X \to f_*\Gm_V \to \widehat{G} \to 1$$
o\`u $\widehat{G} := \mathcal{H}om_{X\textup{-gr}}(G, \Gm_X)$.
\end{prop}

\begin{proof}
La preuve est exactement similaire \`a celle la proposition 1.4.2 de
\cite{CTS}, en utilisant le corollaire VII.1.2 de \cite{Ray}.
\end{proof}

On d\'eduit imm\'ediatement de cette proposition le premier point de la proposition \ref{prop qis caracteres}.

Montrons le second point : il suffit de montrer que le morphisme naturel $s_{\ov{V}/\ov{X}} : \Pic(\ov{G}) \to \Pic'(\ov{V}/\ov{X})$ est un isomorphisme. On commence par le r\'esultat suivant :
\begin{prop}
\label{prop Leray BDCTX}
Soit $k$ un corps de caract\'eristique nulle, $X$ une $k$-vari\'et\'e lisse g\'eom\'etriquement int\`egre. Soit $f : V \to X$ un
torseur sous un $k$-groupe lin\'eaire connexe $G$. On a alors un
diagramme commutatif naturel de suites exactes :
\begin{displaymath}
\xymatrix{
k[G]^* / k^* \ar[r]^{\Delta'_{V/X}} \ar[d] & \Pic(X) \ar[r]^{f^*}
\ar[d]^{t_{V/X}} & \Pic(V) \ar[r]^{\varphi_1} \ar[d]^{=} & \Pic(G)
\ar[r]^{\Delta_{V/X}} \ar[d]^{s_{V/X}} & \Br(X) \ar[r]^{f^*} \ar[d]^{r_{V/X}} &
\Br(V) \ar[d]^= \\
0 \ar[r] & H^1(X, f_*\Gm_V) \ar[r]^{\varphi} & \Pic(V)
\ar[r]^{\gamma} & \Pic'(V/X) \ar[r]^{\delta_{V/X}} & H^2(X, f_* \Gm_V)
\ar[r]^{\beta} & \Br(V) \, ,
}
\end{displaymath}
o\`u la suite sup\'erieure provient de \cite{BD}, th\'eor\`eme 2.8, et la suite inf\'erieure est la suite associ\'ee \`a la suite spectrale de Leray $E_2^{p,q} := H^p(X, R^q f_* \Gm_V) \Rightarrow H^{p+q}(V, \Gm_V)$.
\end{prop}

\begin{proof}
Les morphismes $t_{V/X}$ et $r_{V/X}$ sont induits par le morphisme canonique
$\Gm_X \to f_* \Gm_V$, et le morphisme $s_{V/X}$ est d\'efini au d\'ebut de la section \ref{section W/X}.

Montrons la commutativit\'e du diagramme ainsi construit.
\begin{itemize}
  \item Commutativit\'e du premier carr\'e : il suffit de montrer que
    la compos\'ee 
$$k[G]^* / k^* \xrightarrow{\Delta'_{V/X}} \Pic(X) \xrightarrow{t_{V/X}}
H^1(X, f_*\Gm_V)$$
est le morphisme nul. Ceci est une cons\'equence de la commutativit\'e
du deuxi\`eme carr\'e et de l'exactitude de la ligne inf\'erieure du
diagramme
\item Commutativit\'e du deuxi\`eme carr\'e : c'est un calcul
  classique : voir par exemple \cite{Gir}, proposition 3.1.3 et
  3.1.4.1.
  \item Commutativit\'e du troisi\`eme carr\'e : c'est le lemme \ref{lem diag UPic comm}.
  \item Commutativit\'e du quatri\`eme carr\'e : c'est une cons\'equence
  de la proposition V.3.2.9 de \cite{Gir} et de sa preuve, en
  utilisant la d\'efinition de $\Delta_{V/X}$ (voir \cite{CTX} section 2 p.314
  ou \cite{BD} 2.5).
  \item Commutativit\'e du dernier carr\'e : la preuve est similaire
    \`a celle du deuxi\`eme carr\'e.
\end{itemize}
\end{proof}

On d\'eduit de cette proposition, gr\^ace au lemme des cinq, le corollaire suivant :
\begin{cor}
\label{cor Leray 2}
Sous les hypoth\`eses de la proposition \ref{prop Leray BDCTX}, si on
suppose de plus que $\widehat{G}(k) = k[G]^* / k^*$ est trivial (i.e. $G\tor = 0$), alors
$f_* \Gm_V = \Gm_X$ et le diagramme de la proposition \ref{prop Leray BDCTX} devient :
\begin{displaymath}
\xymatrix{
0 \ar[r] & \Pic(X) \ar[r]^{f^*}
\ar[d]^{t_{V/X}}_{\cong} & \Pic(V) \ar[r]^{\varphi_1} \ar[d]^{=} & \Pic(G)
\ar[r]^{\Delta_{V/X}} \ar[d]^{s_{V/X}}_{\cong} & \Br(X) \ar[r]^{f^*} \ar[d]^{r_{V/X}}_{\cong} &
\Br(V) \ar[d]^= \\
0 \ar[r] & H^1(X, f_*\Gm_V) \ar[r]^{\varphi} & \Pic(V) \ar[r]^{\gamma} & \Pic'(V/X) \ar[r]^{\delta_{V/X}} & H^2(X, f_* \Gm_V)
\ar[r]^{\beta} & \Br(V) \, .
}
\end{displaymath}
\end{cor}

Revenons \`a la preuve de la proposition \ref{prop qis caracteres} : le corollaire \ref{cor Leray 2} assure que le morphisme $s_{\ov{V}/\ov{X}} : \Pic(\ov{G}) \to \Pic'(\ov{V}/\ov{X})$ est un isomorphisme, ce qui conclut la preuve.
%
\end{proof}

\subsubsection{Application au cas d'un torseur avec donn\'ee de recollement}
\label{section Z X torseurs}
Revenons \`a la situation et aux notations introduites au d\'ebut de la section \ref{section donnee descente}.

On applique la proposition \ref{prop qis caracteres} au torseur $p : Z
\xrightarrow{H'} X$. Les hypoth\`eses de cette proposition sont
v\'erifi\'ees, puisque ${H'}\tor = 0$. Par cons\'equent, on a
l'\'enonc\'e suivant :
\begin{prop}
\label{prop Z X torseurs}
Soit $\ov{Y}/X$ un torseur sous $\ov{H}$ lin\'eaire connexe, muni d'une donn\'ee de
 recollement. On conserve les notations de la section \ref{section donnee descente}.
Alors on a un quasi-isomorphisme 
$$C_{\ov{Y}/X} = C_{\ov{Z}/X} \to \UPic(p : Z \to X) \, ,$$
et un isomorphisme
$$\Gm_X \cong p_* \Gm_Z \, .$$
\end{prop}

\section{Groupe de Brauer d'un torseur avec donn\'ee de recollement}
\label{section Brauer torseurs descente}
Dans cette section, le cadre est le m\^eme que dans la pr\'ec\'edente : $k$
est un corps de caract\'eristique nulle, $X$ est une $k$-vari\'et\'e
lisse g\'eom\'etriquement int\`egre, et $\ov{H}$ est un $\ov{k}$-groupe alg\'ebrique. Soit $\ov{\pi} : \ov{Y}
\xrightarrow{\ov{H}} \ov{X}$ un $\ov{X}$-torseur sous $\ov{H}$ muni d'une donn\'ee de recollement. On d\'efinit le groupe
$$\Br_1(X,\ov{Y}) := \textup{Ker}(\Br(X) \to \Br(\ov{X}) \xrightarrow{\ov{\pi^*}} \Br(\ov{Y}))$$ 
comme dans \cite{BD}, et $\Br_a(X,\ov{Y}) := \Br_1(X,\ov{Y})/\Br(k)$. On l'appelle groupe de Brauer alg\'ebrique du ``torseur'' $\ov{Y}/X$. Le r\'esultat principal de cette section est le suivant :

\begin{theo}
\label{theo torseurs}
Avec les notations pr\'ec\'edentes, on dispose d'un isomorphisme canonique, fonctoriel en
$(X,\ov{Y},\ov{H},\ov{\pi})$ :
$$U(X) \xrightarrow{\cong} \H^0(k, C_{\ov{Y}/X}) \, ,$$
et d'une suite exacte de groupes commutatifs, fonctorielle en $(X,\ov{Y},\ov{H},\ov{\pi})$ :
\begin{equation*}
0 \to \Pic(X) \to \H^1(k, C_{\ov{Y}/X}) \to \Br(k) \to \Br_1(X,\ov{Y})
 \to \H^2(k,C_{\ov{Y}/X}) \to N^3(k, \Gm) \, ,
\end{equation*}
o\`u $N^3(k, \Gm) := \textup{Ker}(H^3(k, \Gm) \to H^3(X, \Gm))$.
\end{theo}

Avant de commencer la preuve, rappelons la suite exacte (\ref{sel base}) dans le contexte du morphisme $p : Z \to X$ induit par la donn\'ee de recollement :
\begin{equation}
\label{sel base Z}
0 \to \Gm_X \to {p}_* \mathcal{K}_Z^* \to {p}_* \mathcal{D}iv_Z \to R^1 p_* \Gm_Z \to 0 \, ,
\end{equation}
et d\'efinissons le morphisme $\partial_2 : {p_X}_* R^1 p_* \Gm_Z \to R^2 {p_X}_* \Gm_X$ comme le cobord it\'er\'e associ\'e au foncteur ${p_X}_*$ et \`a cette suite exacte (rappelons que la proposition \ref{prop qis caracteres} permet d'identifier $\Gm_X = {p}_* \Gm_Z$).

\begin{proof}
%
%
D\'efinissons $\UPic'(p : Z \to X) := [\ov{k}(Z)^* \to \textup{Div}(\ov{Z}) \xrightarrow{\varphi'} \Pic'(\ov{Z}/\ov{X})]$. 
\begin{prop}
\label{prop qis Brauer}
Il existe un triangle exact naturel
$$\UPic'(p) \to \tau_{\leq 2} \R {p_X}_* \Gm_X \to \left( R^2 {p_X}_* \Gm_X / \partial_2({p_X}_* R^1 p_* \Gm_Z) \right)[-2] \to \UPic'(p)[1] \, ,$$
et une suite exacte canonique
$$0 \to \partial_2({p_X}_* R^1 p_* \Gm_Z) \to R^2 {p_X}_* \Gm_X \to R^2 {p_Z}_* \Gm_Z \, .$$
\end{prop}

\begin{proof}
On montre en fait le r\'esultat plus g\'en\'eral suivant :
\begin{lem}
\label{lem categ degre 2}
Soient $\mathcal{A}$, $\mathcal{B}$ deux cat\'egories ab\'eliennes, et
$F : \mathcal{A} \to \mathcal{B}$ un foncteur exact \`a gauche.
Soit
\begin{equation}
\label{suite longue derivee}
0 \to A \to B_0 \to B_1 \to \dots \to B_n \to 0
\end{equation}
une suite exacte dans $\mathcal{A}$. On suppose que $\mathcal{A}$ a
suffisamment d'injectifs et que $(R^i F) B_k = 0$ pour tout $0 \leq k \leq n-2$ et tout $1 \leq i \leq n-1-k$. Alors on a un triangle exact canonique (dans la cat\'egorie d\'eriv\'ee associ\'ee \`a la cat\'egorie des complexes born\'es de $\mathcal{A}$) :
$$R^0 F [B_0 \to \dots \to B_n] \to \tau_{\leq n} (\R F) A \to \left((R^n F) A / \partial_n(B_n)\right)[-n] \to R^0 F [B_0 \to \dots \to B_n][1] \, ,$$
o\`u $\partial_n : B_n \to (R^n F) A$ est le cobord it\'er\'e associ\'e \`a la suite exacte (\ref{suite longue derivee}).
Supposons en outre que $(R^{n-1} F) B_1 = (R^{n-2} F) B_2 = \dots = (R^1 F) B_{n-1} = 0$, alors on a une suite exacte canonique
$$0 \to \partial_n(B_n) \to (R^n F) A \to (R^n F) B_0 \, .$$
\end{lem}

\begin{proof}
On note $B_{\bullet}$ le complexe
$$B_{\bullet} := [B_0 \to B_1 \to \dots \to B_n]\, .$$
Soit $B_{\bullet} \xrightarrow{\epsilon_{\bullet}} I_{\bullet, \bullet}$
 une r\'esolution injective de Cartan-Eilenberg de $B_{\bullet}$ (voir
 \cite{Wei}, 5.7.9). Consid\'erons le complexe total
 $\textup{Tot}(F(I_{\bullet, \bullet}))$ associ\'e au complexe double
 $F(I_{\bullet, \bullet})$. On a un morphisme injectif naturel de complexes de longueur $n+1$ 
$$F(\epsilon_{\bullet}) : F(B_{\bullet}) \to \tau_{\leq n} \textup{Tot}(F(I_{\bullet, \bullet})) \, .$$
Montrons que son conoyau $Q_{\bullet}$ est quasi-isomorphe \`a $\left(\mathcal{H}^n(\textup{Tot}(F(I_{\bullet, \bullet}))) / \epsilon_n(B_n)\right)[-n]$.
Par d\'efinition, pour tout $0 \leq k \leq n-1$, on a 
$$Q_k = F(I_{0,k}) \oplus \dots \oplus F(I_{k-1,1}) \oplus \left( F(I_{k,0}) / \epsilon_k(F(B_k)) \right) \, .$$
Fixons $0 \leq k \leq n-1$. Puisque $(R^k F) B_0 = (R^{k-1} F) B_1 = \dots = (R^1 F) B_{k-1} = 0$, une chasse au diagramme simple dans $F(I_{\bullet, \bullet})$ assure que $Q_{\bullet}$ est exact en degr\'e $k$. 
Par cons\'equent, on a un quasi-isomorphisme canonique
$$Q_{\bullet} \to \mathcal{H}^n(Q_{\bullet})[-n] = \left( Q_n / Q_{n-1} \right)[-n] \, .$$
Or $Q_n$ est le quotient du noyau de
$$F(I_{0,n}) \oplus \dots \oplus F(I_{n,0}) \to F(I_{0,n+1}) \oplus \dots \oplus F(I_{n+1,0})$$
par $\epsilon_n(B_n) \subset F(I_{n,0})$, donc 
$$Q_n / Q_{n-1} = \mathcal{H}^n(\textup{Tot}(F(I_{\bullet, \bullet}))) / \epsilon_n(B_n) \, .$$
Finalement, la suite exacte de complexes
$$0 \to F(B_{\bullet}) \xrightarrow{\epsilon_{\bullet}} \tau_{\leq n} \textup{Tot}(F(I_{\bullet, \bullet})) \to Q_{\bullet} \to 0$$
induit un triangle exact canonique (qui ne d\'epend pas de la r\'esolution $I_{\bullet, \bullet}$ choisie) dans la cat\'egorie d\'eriv\'ee
\begin{equation}
\label{triangle derive resolution}
F(B_{\bullet}) \to \tau_{\leq n} (\R F) B_{\bullet} \to \left( (R^n F) B_{\bullet} / \epsilon_n(B_n) \right)[-n] \to F(B_{\bullet})[1] \, .
\end{equation}
Enfin la suite exacte (\ref{suite longue derivee}) induit un isomorphisme naturel dans la cat\'egorie d\'eriv\'ee
$$A \cong B_{\bullet} $$
qui permet d'identifier le triangle exact (\ref{triangle derive resolution}) au triangle exact du lemme (apr\`es avoir identifi\'e $\epsilon_n(B_n) \subset (R^n F) B_{\bullet}$ avec $\partial_n(B_n) \subset (R^n F) A$, en s'inspirant de \cite{CE}, proposition 7.1).

Enfin, la derni\`ere partie du lemme est \'evidente en d\'ecomposant la suite (\ref{suite longue derivee}) en suites exactes courtes successives et \'ecrivant les suites exactes longues associ\'ees.
\end{proof}


Appliquons le lemme \ref{lem categ degre 2} \`a la
situation consid\'er\'ee dans la proposition \ref{prop qis Brauer} et \`a la suite exacte (\ref{sel base Z}) dans la cat\'egorie $\mathcal{A}$ des
faisceaux ab\'eliens fppf sur $X$, $\mathcal{B}$ \'etant la cat\'egorie des faisceaux ab\'eliens fppf sur $\textup{Spec}(k)$ et le foncteur $F$ \'etant l'image directe $(p_X)_*$ par le morphisme structural $p_X : X \to \textup{Spec}(k)$.
Il faut v\'erifier les hypoth\`eses suivantes pour appliquer le lemme : $(R^1 {p_X}_*) p_* \mathcal{K}_Z^* = 0$ et $(R^1 {p_X}_*) p_*\mathcal{D}iv_Z = 0$. La premi\`ere annulation est une cons\'equence de Hilbert 90, puisque $(R^1 {p_X}_*) p_* \mathcal{K}_Z^*$ s'injecte dans $(R^1 {p_Z}_*) \mathcal{K}_Z^*$, qui est nul par Hilbert 90 (voir \cite{GBr} II, lemme 1.6). 
La seconde annulation provient de l'injection naturelle $(R^1 {p_X}_*) p_* \mathcal{D}iv_Z \to (R^1 {p_Z}_*)\mathcal{D}iv_Z = 0$ (puisque $Z$ est lisse, le faisceau $\mathcal{D}iv_Z$ s'identifie au faisceau des diviseurs de Weil : voir \cite{GBr} II).

Ainsi le lemme \ref{lem categ degre 2} fournit-il un triangle exact naturel dans la cat\'egorie d\'eriv\'ee
$$R^0 {p_X}_* [p_* \mathcal{K}_Z^* \to p_* \mathcal{D}iv_Z \to R^1 p_* \Gm_Z] \to \tau_{\leq 2} \R {p_X}_* \Gm_X \to $$
$$\to \left( R^2 {p_X}_* \Gm_X / \partial_2({p_X}_* R^1 p_* \Gm_Z) \right)[-2] \to R^0 {p_X}_* [p_* \mathcal{K}_Z^* \to p_* \mathcal{D}iv_Z \to R^1 p_* \Gm_Z][1] \, ,$$
et une suite exacte 
\begin{equation}
\label{sec version 1}
0 \to \partial_2({p_X}_* R^1 p_* \Gm_Z) \to R^2 {p_X}_* \Gm_X \to R^2 {p_X}_* p_* \mathcal{K}_Z^* \, .
\end{equation}

Or $R^0 {p_X}_* [p_* \mathcal{K}_Z^* \to p_* \mathcal{D}iv_Z \to
R^1 p_* \Gm_Z] = \UPic'(p : Z \to X)$, donc le triangle exact devient
\begin{equation}
\label{triangle exact final}
\UPic'(p) \to \tau_{\leq 2} \R {p_X}_* \Gm_X \to \left( R^2 {p_X}_* \Gm_X / \partial_2({p_X}_* R^1 p_* \Gm_Z) \right)[-2] \to \UPic'(p)[1] \, .
\end{equation}

On consid\`ere alors le diagramme commutatif naturel suivant :
\begin{displaymath}
\xymatrix{
(R^2 {p_X}_*) p_* \Gm_Z \ar[r] \ar[d] & (R^2 {p_X}_*) p_*
\mathcal{K}_Z^* \ar[d] \\
(R^2 {p_Z}_*) \Gm_Z \ar[r] & (R^2 {p_Z}_*) \mathcal{K}_Z^* \, ,
}
\end{displaymath}
o\`u les fl\`eches verticales proviennent des suites spectrales
\'evidentes. La trivialit\'e des faisceaux $(R^1 {p_Z}_*)
\mathcal{D}iv_Z$ et ${p_X}_* (R^1 p_*) \mathcal{K}_Z^*$ assure que le
morphisme du bas et celui de droite sont injectifs, ce qui implique
que l'on a une \'egalit\'e
$$\textup{Ker}((R^2 {p_X}_*) p_* \Gm_Z \to (R^2 {p_X}_*) p_*
\mathcal{K}_Z^*) = \textup{Ker}((R^2 {p_X}_*) p_* \Gm_Z \to (R^2
{p_Z}_*) \Gm_Z) \, ,$$
donc la suite exacte (\ref{sec version 1}) se r\'e\'ecrit
\begin{equation}
\label{sec version 2}
0 \to \partial_2({p_X}_* R^1 p_* \Gm_Z) \to R^2 {p_X}_* \Gm_X \to R^2 {p_Z}_* \Gm_Z \, .
\end{equation}
Finalement, le triangle (\ref{triangle exact final}) et la suite (\ref{sec version 2}) terminent la preuve de la proposition \ref{prop qis Brauer}.
\end{proof}

Poursuivons la preuve du th\'eor\`eme \ref{theo torseurs}. On consid\`ere le triangle exact \'evident :
\begin{equation}
\label{triangle exact evident}
\Gm \to \UPic'(p : Z \to X) \to \UPic(p : Z \to X) \to \Gm[1] \, .
\end{equation}
La proposition \ref{prop Z X torseurs} assure que l'on a un isomorphisme canonique dans la cat\'egorie d\'eriv\'ee :
$$C_{\ov{Y}/X} \cong \UPic(p : Z \to X) \, .$$
Donc la suite exacte longue associ\'ee au triangle (\ref{triangle exact evident}) fournit les suites exactes suivantes :
\begin{equation}
\label{sel U}
0 \to H^0(k, \Gm) \to \H^0(k, \UPic'(p)) \to \H^0(k, C_{\ov{Y}/X}) \to H^1(k, \Gm) = 0
\end{equation}
et
\begin{align}
\label{sel PicBr}
& 0 \to \H^1(k, \UPic'(p)) \to \H^1(k, C_{\ov{Y}/X}) \to H^2(k, \Gm)
\to \\
& \to \H^2(k, \UPic'(p)) \to \H^2(k, C_{\ov{Y}/X}) \to H^3(k, \Gm) \to
\H^3(k, \UPic'(p)) \nonumber \, .
\end{align}

Or la proposition \ref{prop qis Brauer} assure que l'on a un triangle exact canonique :
\begin{equation}
\label{triangle exact prime}
\UPic'(p) \to \tau_{\leq 2} \R {p_X}_* \Gm_X \to \left( R^2 {p_X}_*
						  \Gm_X /
						  \partial_2({p_X}_*
						  R^1p_* \Gm_Z)
						 \right)[-2] \to \UPic'(p)[1] \, .
\end{equation}
Ce triangle induit des isomorphismes canoniques :
$$\H^0(k, \UPic'(p)) = k[X]^* \, , \, \, \, \H^1(k, \UPic'(p)) = \Pic(X) \, ,$$
ainsi que l'isomorphisme suivant (en utilisant la seconde partie de la proposition \ref{prop qis Brauer}) :
$$\H^2(k, \UPic'(p)) = \textup{Ker}(\Br(X) \xrightarrow{p^*} \Br(\ov{Z})) \, .$$

En combinant ces isomorphismes avec les suites exactes (\ref{sel U}) et (\ref{sel PicBr}), on obtient le th\'eor\`eme \ref{theo torseurs}, \`a ceci pr\`es qu'il reste \`a 
%
%
%
%
%
%
%
identifier le groupe $\textup{Ker}(\Br(X) \xrightarrow{p^*}
 \Br(\ov{Z}))$ au groupe $\Br_1(X,\ov{Y}) = \textup{Ker}(\Br(X) \to
 \Br(\ov{X}) \xrightarrow{\ov{\pi}^*} \Br(\ov{Y}))$, et \`a comparer le groupe
 $\textup{Ker}(H^3(k, \Gm) \to \H^3(k, \UPic'(p)))$ au groupe
 $\textup{Ker}(H^3(k, \Gm) \to H^3(X, \Gm))$. 

On dispose d'abord du fait suivant :
\begin{lem}
\label{lem surj Pic}
Le morphisme $\Pic(\ov{H'}) \rightarrow \Pic(\ov{H})$ est surjectif.
\end{lem}

\begin{proof}
On note $\ov{H}\red$ le quotient de $\ov{H}$ par son radical unipotent. 
Le morphisme $\Pic(\ov{H}\red) \to \Pic(\ov{H})$ est surjectif car le groupe de Picard du radical unipotent est trivial. Le morphisme $\ov{H}\red \to \ov{H}'$ est un morphisme surjectif, \`a noyau diagonalisable, entre groupes connexes, par cons\'equent la proposition 4.2 de \cite{FI} assure que le morphisme $\Pic(\ov{H}') \to \Pic(\ov{H}\red)$ est surjectif. D'o\`u le lemme.
%
%
\end{proof}

Consid\'erons le diagramme commutatif exact suivant :
\begin{displaymath}
\xymatrix{
& \Pic(\ov{H'}) \ar[d]^{\Delta_{\ov{Z}/\ov{X}}} \ar[ld] & \\
\Pic(\ov{H}) \ar[r]_{\Delta_{\ov{Y}/\ov{X}}} & \Br(\ov{X}) \ar[r]^{\pi^*} \ar[d]^{p^*} & \Br(\ov{Y}) \ar[d]^= \\
 & \Br(\ov{Z}) \ar[r]^{{\pi'}^*} & \Br(\ov{Y}) \, .
}
\end{displaymath}
On d\'eduit alors du lemme \ref{lem surj Pic}, par une chasse au diagramme facile, que l'inclusion naturelle 
$$\textup{Ker}(\Br(X) \xrightarrow{p^*} \Br(\ov{Z})) \subset
\Br_1(X,\ov{Y})$$
est en fait une \'egalit\'e. 

Pour conclure la preuve du th\'eor\`eme \ref{theo torseurs}, montrons
 que l'on a un morphisme injectif 
$$\textup{Ker}(H^3(k,
 \Gm) \to \H^3(k, \UPic'(p))) \hookrightarrow \textup{Ker}(H^3(k, \Gm)
 \to H^3(X, \Gm)) \, .$$
Pour cela, on consid\`ere le carr\'e commutatif suivant, o\`u le
 morphisme horizontal inf\'erieur est extrait du triangle exact (\ref{triangle exact prime}) et les
 morphismes verticaux sont les morphismes \'evidents :
\begin{displaymath}
\xymatrix{
\Gm \ar[r]^= \ar[d] & \Gm \ar[d] \\
\UPic'(p) \ar[r] & \tau_{\leq 2} \R {p_X}_* \Gm_X \, .
}
\end{displaymath}
Or $\H^3(k, \tau_{\leq 2} \R {p_X}_* \Gm_X) \cong \textup{Ker}(H^3(X,
 \Gm) \to H^3(\ov{X}, \Gm))$ (le calcul est le m\^eme que celui de
 \cite{BvH}, 2.18), donc on d\'eduit imm\'ediatement de ce diagramme une inclusion
$$\textup{Ker}(H^3(k, \Gm) \to \H^3(k, \UPic'(p))) \subset \textup{Ker}(H^3(k, \Gm) \to H^3(X, \Gm)) \, ,$$
ce qui termine la preuve du th\'eor\`eme \ref{theo torseurs}.
\end{proof}

\begin{cor}
\label{cor torseurs}
On conserve les hypoth\`eses du th\'eor\`eme \ref{theo torseurs}.
\begin{itemize} 
\item Supposons que $Z(k) \neq \emptyset$. Alors on a trois isomorphismes canoniques
$$U(X) \cong \H^0(k, C_{\ov{Y}/X}) \, , \, \, \, \Pic(X) \cong \H^1(k, C_{\ov{Y}/X}) \, , \, \, \, \Br_a(X,Y) \cong \H^2(k, C_{\ov{Y}/X}) \, .$$
\item Supposons que $\Br(k)$ s'injecte dans $\Br(X)$. Alors on a des isomorphismes
$$U(X) \cong \H^0(k, C_{\ov{Y}/X}) \, , \, \, \, \Pic(X) \cong \H^1(k, C_{\ov{Y}/X}) \, , $$
et une suite exacte 
$$0 \to \Br_a(X, \ov{Y}) \to \H^2(k, C_{\ov{Y}/X}) \to
      \textup{Ker}(H^3(k, \Gm) \to H^3(X, \Gm)) \, .$$
\item Supposons que $H^3(k, \Gm)$ s'injecte dans $H^3(X, \Gm)$. Alors on a des isomorphismes
$$U(X) \cong \H^0(k, C_{\ov{Y}/X}) \, , \, \, \, \Br_a(X,\ov{Y}) \cong \H^2(k, C_{\ov{Y}/X}) \, , $$
et une suite exacte 
$$0 \to \Pic(X) \to \H^1(k, C_{\ov{Y}/X}) \to \textup{Ker}(\Br(k) \to \Br(X)) \, .$$
\end{itemize}
\end{cor}

\begin{proof}
\begin{itemize}
  \item On suppose que $Z(k) \neq \emptyset$. Alors comme dans la section 2.8 de \cite{BvH}, un point $z \in Z(k)$ d\'efinit une section du triangle exact (\ref{triangle exact evident}) :
$$ \Gm \to \UPic'(p : Z \to X) \to \UPic(p : Z \to X) \to \Gm[1] \, .$$
L'existence de cette section assure que les morphismes $\H^1(k, C_{\ov{Y}/X}) \to \Br(k)$ et $\H^2(k, C_{\ov{Y}/X}) \to H^3(k, \Gm)$ dans la suite exacte du th\'eor\`eme \ref{theo torseurs} sont des morphismes nuls, d'o\`u le premier point du corollaire.
\item Le th\'eor\`eme \ref{theo torseurs} assure
      imm\'ediatement les deux autres points du corollaire.
\end{itemize}
\end{proof}



\section{Groupe de Brauer alg\'ebrique d'un torseur}
\label{section torseurs}
Dans cette section, on applique les r\'esultats g\'en\'eraux de la section \ref{section Brauer torseurs descente} \`a la situation d'un torseur $\pi : Y \xrightarrow{H} X$ d\'efini sur $k$. Ceci est clairement un cas particulier de torseur avec donn\'ee de recollement. On d\'eduit donc des sections pr\'ec\'edentes les r\'esultats suivants :
\begin{theo}
\label{theo torseurs k}
Soit $k$ un corps de caract\'eristique nulle, $X$ une $k$-vari\'et\'e lisse g\'eom\'etriquement int\`egre, $H$ un $k$-groupe lin\'eaire connexe et $\pi : Y \xrightarrow{H} X$ un $X$-torseur sous $H$. On a alors un isomorphisme canonique, fonctoriel en
$(X,Y,H,\pi)$ :
$$U(X) \xrightarrow{\cong} \H^0(k, C_{\ov{Y}/X}) \, ,$$
et une suite exacte de groupes commutatifs, fonctorielle en $(X, Y, H, \pi)$ :
$$0 \to \Pic(X) \to \H^1(k, C_{\ov{Y}/X}) \to \Br(k) \to \Br_1(X,Y) \to
 \H^2(k,C_{\ov{Y}/X}) \to N^3(k, \Gm) \, ,$$
o\`u $N^3(k, \Gm) := \textup{Ker}(H^3(k, \Gm) \to H^3(X, \Gm))$.
\end{theo}

On en d\'eduit aussi les corollaires suivants. Le premier est \'evident :
\begin{cor}
Avec les notations du th\'eor\`eme \ref{theo torseurs k} :
\begin{itemize} 
\item
supposons que $\Br(k)$ s'injecte dans $\Br(X)$. Alors on a des isomorphismes
$$U(X) \xrightarrow{\cong} \H^0(k, C_{\ov{Y}/X}) \, , \, \, \, \Pic(X) \xrightarrow{\cong} \H^1(k, C_{\ov{Y}/X}) \, ,$$
et une suite exacte
$$0 \to \Br_a(X,Y) \to \H^2(k, C_{\ov{Y}/X}) \to \textup{Ker}(H^3(k,\Gm)
     \to \H^3(X, \Gm)) \, .$$

\item supposons que $H^3(k, \Gm)$ s'injecte dans $H^3(X, \Gm)$. Alors on a des isomorphismes
$$U(X) \xrightarrow{\cong} \H^0(k, C_{\ov{Y}/X}) \, , \, \, \, \Br_a(X,Y) \xrightarrow{\cong} \H^2(k, C_{\ov{Y}/X}) \, ,$$
et une suite exacte
$$0 \to \Pic(X) \to \H^1(k, C_{\ov{Y}/X}) \to \textup{Ker}(\Br(k) \to \Br(X)) \, .$$
\end{itemize}
\end{cor}

Le second concerne les torseurs munis d'un point rationnel (voir \cite{BvH} 2.8 pour les notations $\ov{k}(Y)^*_{y,1}$ et $\textup{Div}(\ov{Y})_y$) :
\begin{cor}
\label{coro torseurs pt ratio}
Avec les notations du th\'eor\`eme \ref{theo torseurs k}, supposons qu'il existe $y \in Y(k)$. Alors on a trois isomorphismes
$$U(X) \xrightarrow{\cong} \H^0(k, C_{\ov{Y}/X}) \, , \, \, \,  \Pic(X) \xrightarrow{\cong} \H^1(k, C_{\ov{Y}/X}) \, , \, \, \, \Br_a(X,Y) \xrightarrow{\cong} \H^2(k, C_{\ov{Y}/X}) \, .$$
En outre, le complexe $C_{\ov{Y}/X}$ est canoniquement quasi-isomorphe au complexe de modules galoisiens
$$\textup{C\^one}\left([\ov{k}(Y)^*_{y,1} \to \textup{Div}(\ov{Y})_y] \xrightarrow{i_y^*} [\ov{k}(H)^*_{e,1} \to \textup{Div}(\ov{H})_e]\right) \, ,$$
(o\`u $i_y : H \to Y$ est d\'efinie par $h \mapsto y.h$), i.e. \`a 
$$[\ov{k}(Y)^*_{y,1} \to \ov{k}(H)^*_{e,1} \oplus \textup{Div}(\ov{Y})_y \to \textup{Div}(\ov{H})_e]  \, .$$
\end{cor}

\begin{proof}
La seule chose \`a d\'emontrer est le quasi-isomorphisme 
$$C_{\ov{Y}/X} \to C_{\ov{Y}/X, y} := \textup{C\^one}\left([\ov{k}(Y)^*_{y,1} \to \textup{Div}(\ov{Y})_y] \xrightarrow{i_y^*} [\ov{k}(H)^*_{e,1} \to \textup{Div}(\ov{H})_e]\right) \, .$$
Pour ce faire, il est clair (voir \cite{BvH} 2.8) que $C_{\ov{Y}/X}$ est quasi-isomorphe \`a 
$$C_{\ov{Z}/X, z} := \textup{C\^one}\left([\ov{k}(Z)^*_{z,1} \to \textup{Div}(\ov{Z})_z] \xrightarrow{i_z^*} [\ov{k}(H')^*_{e,1} \to \textup{Div}(\ov{H'})_e] \right) \, ,$$
o\`u $z \in Z(k)$ est l'image de $y$.
Or on a un diagramme commutatif naturel de complexes de modules galoisiens :
\begin{displaymath}
\xymatrix{
[\ov{k}(Z)^*_{z,1} \to \textup{Div}(\ov{Z})_z] \ar[r] \ar[d] & [\ov{k}(H')^*_{e,1} \to \textup{Div}(\ov{H'})_e] \ar[d] \\
[\ov{k}(Y)^*_{y,1} \to \textup{Div}(\ov{Y})_y] \ar[r] & [\ov{k}(H)^*_{e,1} \to \textup{Div}(\ov{H})_e] \, 
}
\end{displaymath}
dont on d\'eduit un morphisme canonique de complexes de modules galoisiens
$$\alpha : C_{\ov{Z}/X,z} \to C_{\ov{Y}/X,y} \, .$$
Montrons que ce dernier est un quasi-isomorphisme. Pour $i= 0, 1, 2$, on d\'efinit 
$$f_i := \mathcal{H}^i(\alpha) : \mathcal{H}^i(C_{\ov{Z}/X,z}) \to \mathcal{H}^i(C_{\ov{Y}/X,y}) \, .$$
Tout d'abord, le morphisme $f_0$ s'int\`egre dans le diagramme commutatif exact suivant :
\begin{displaymath}
\xymatrix{
0 \ar[r] & \mathcal{H}^0(C_{\ov{Z}/X,z}) \ar[r] \ar[d]^{f_0} & U(\ov{Z}) \ar[r]
\ar[d] & U(\ov{H'}) \ar[d] \\
0 \ar[r] & \mathcal{H}^0(C_{\ov{Y}/X,y}) \ar[r] & U(\ov{Y}) \ar[r] & U(\ov{H}) \, ,
}
\end{displaymath}
ce qui assure (en utilisant par exemple la proposition 6.10 de \cite{San}) que $f_0$ est un isomorphisme qui identifie
$\mathcal{H}^0(C_{\ov{Z}/X,z})$ et $\mathcal{H}^0(C_{\ov{Y}/X,y})$ \`a $U(\ov{X})$.

En degr\'e $2$, il est clair que le morphisme $f_2$ s'identifie au morphisme canonique 
$$\Pic(\ov{H}') / \Pic(\ov{Z}) \to \Pic(\ov{H}) / \Pic(\ov{Y}) \, .$$
Or on a un diagramme commutatif de suites exactes (voir \cite{San}, proposition 6.10 ou \cite{BD}, th\'eor\`eme 2.8) :
\begin{displaymath}
\xymatrix{
\Pic(\ov{Z}) \ar[r] \ar[d] & \Pic(\ov{H}') \ar[r] \ar[d] & \Br(\ov{X}) \ar[d]^= \\
\Pic(\ov{Y}) \ar[r] & \Pic(\ov{H}) \ar[r] & \Br(\ov{X}) \, ,
}
\end{displaymath}
et le lemme \ref{lem surj Pic} assure que le morphisme central est surjectif. Une chasse au diagramme assure alors que le morphisme 
$$\Pic(\ov{H}') / \Pic(\ov{Z}) \to \Pic(\ov{H}) / \Pic(\ov{Y})$$
est un isomorphisme


De la m\^eme fa\c con, en degr\'e $1$, la cohomologie des deux complexes $C_{\ov{Y}/X,y}$ et $C_{\ov{Z}/X,z}$ est exactement $\Pic(\ov{X})$ et le morphisme $f_1$ est l'identit\'e de $\Pic(\ov{X})$.

Finalement, tous les $f_i$ sont des isomorphismes, donc $\alpha$ est un quasi-isomorphisme.
\end{proof}

\begin{exs}
Dans la situation du th\'eor\`eme \ref{theo torseurs k}, l'objet $C_{\ov{Y}/X}$ est parfois repr\'esentable par un complexe explicite de modules galoisiens d\'efinis \`a partir de $Y$ et $H$ (sans faire intervenir $Z$), comme le montrent les exemples suivants :
\begin{itemize}
\item Supposons qu'il existe $y \in Y(k)$. Alors on a montr\'e au
      corollaire \ref{coro torseurs pt ratio} l'existence d'un quasi-isomorphisme canonique :
$$C_{\ov{Y}/X} \rightarrow [\ov{k}(Y)^*_{y,1} / \ov{k}^* \to \textup{Div}(\ov{Y})_y \oplus \ov{k}(H)^*_{e,1}
/ k^* \to \textup{Div}(\ov{H})_e] \, .$$
\item Supposons que $H\tor = 0$. Alors on a un quasi-isomorphisme canonique :
$$C_{\ov{Y}/X} \to [\ov{k}(Y)^*/\ov{k}^* \to \textup{Div}(\ov{Y})
\xrightarrow{\varphi_1} \Pic(\ov{H})] \, .$$
\end{itemize}
\end{exs}

\section{Groupe de Brauer des espaces homog\`enes}
Dans cette section, on donne une autre application des r\'esultats g\'en\'eraux de la section \ref{section Brauer torseurs descente}, application aux espaces homog\`enes \`a stabilisateurs lin\'eaires connexes, ne poss\'edant pas n\'ecessairement de point rationnel.

Consid\`erons un groupe alg\'ebrique connexe $G$ sur un corps $k$ de caract\'eristique nulle, et un espace homog\`ene (\`a gauche) $X$ de $G$ sur $k$. Si on choisit $\ov{x} \in X(\ov{k})$, on obtient un morphisme
$$\ov{\pi}_{\ov{x}} : \ov{G} \xrightarrow{\ov{H}} \ov{X}$$
d\'efini par $\ov{\pi}_{\ov{x}}(\ov{g}) := \ov{g}.\ov{x}$, qui est un torseur (\`a droite) sous le $\ov{k}$-groupe lin\'eaire $\ov{H} := \textup{Stab}_{\ov{G}}(\ov{x})$. 
Sur $\ov{k}$, on dispose d'un diagramme commutatif naturel :
\begin{displaymath}
\xymatrix{
\ov{G} \ar[r]^{\ov{H}\uu} \ar[ddr]_{\ov{H}} & \ov{Z'} \ar[dd]^{\ov{H}\red} \ar[rd]^{Z_{\ov{H}\red}} & \\
& & \ov{Z} \ar[ld]^{\ov{H'}} \\
& \ov{X} & \, .
}
\end{displaymath}
Remarquons que $\ov{G}$ est muni d'une structure de $\ov{X}$-torseur \`a droite sous $\ov{H}$ et d'une action \`a gauche de $\ov{G}$. Cela induit sur $\ov{Z}$ une structure de $\ov{X}$-torseur \`a droite sous $\ov{H'}$ et une action \`a gauche de $\ov{G}$.

La donn\'ee de $\ov{x}$ d\'efinit une donn\'ee de recollement sur $\ov{\pi}_{\ov{x}} : \ov{G} \to \ov{X}$ (voir \cite{HSk}, proposition 3.3), i.e. une extension localement scind\'ee de groupes topologiques
$$1 \to \ov{H}(\ov{k}) \to \textup{SAut}_G(\ov{G} / X) \to \Gamma_k \to 1$$
v\'erifiant les conditions de la d\'efinition \ref{defi rec} et de la d\'efinition 2.1 de \cite{HSk}. Par fonctorialit\'e, cette suite induit une suite exacte
\begin{equation}
\label{donnee desc esp hom}
1 \to \ov{H'}(\ov{k}) \to \textup{SAut}_G(\ov{Z} / X) \to \Gamma_k \to 1
\end{equation}
o\`u $\textup{SAut}_G(\ov{Z} / X)$ d\'esigne le sous-groupe de $\textup{SAut}(\ov{Z} / X)$ form\'e des $\varphi$ tels que
$$\varphi(g.z) = (^{q(\varphi)}g).\varphi(z) \, ,$$
($q : \textup{SAut}(\ov{Z} / X) \to \Gamma_k$ est le morphisme \'evident et l'action de $\ov{G}$ sur $\ov{Z}$ est celle mentionn\'ee plus haut).

On est donc dans le cadre g\'en\'eral d\'evelopp\'e \`a la section \ref{section Brauer torseurs descente}. Remarquons que le fait que la classe de la donn\'ee de recollement sur $\ov{Z}/X$ soit neutre (i.e. que l'extension (\ref{donnee desc esp hom}) soit scind\'ee) assure non seulement que le torseur $\ov{Z} / \ov{X}$ descende en $Z / X$, mais aussi que l'action de $\ov{G}$ sur $\ov{Z}$ descende en une action de $G$ sur $Z$ compatible \`a l'action sur $X$, faisant de $Z$ un espace homog\`ene de $G$ \`a stabilisateur $\textup{Ker}(\ov{H} \to \ov{H'})$.

Dans la suite, on sera amen\'e \`a faire l'une ou l'autre des hypoth\`eses suivantes :
\begin{enumerate}
  \item L'espace homog\`ene $X$ a un point rationnel, i.e. $X(k) \neq \emptyset$.
    \item Le groupe $G$ est lin\'eaire et $\Pic(\ov{G}) = 0$, i.e. $G\abvar = 0$ et $G\ss$ est simplement connexe.
\end{enumerate}

Sous l'une ou l'autre de ces hypoth\`eses, on va associer \`a $X$ (\`a
la paire $(H,G)$) un complexe de modules galoisiens de longueur $3$ not\'e $\widehat{C}_X := [M_1 \to M_2 \to M_3]$
(par convention, $M_1$ est en degr\'e $0$ et $M_3$ en degr\'e $2$)
et construire un morphisme canonique de groupes ab\'eliens $\kappa_X : \H^2(k,\widehat{C}_X) \to \Br_a(X,G)$. Pour ce faire, on va relier le
complexe $\widehat{C}_X$ au complexe $C_{\ov{G}/X}$ (cf section
\ref{section donnee descente}), puis appliquer les r\'esultats
g\'en\'eraux de la section \ref{section Brauer torseurs descente}.

\subsection{Le complexe associ\'e \`a un espace homog\`ene}

\subsubsection{Complexe associ\'e \`a un groupe alg\'ebrique connexe}
On commence par construire la sous-vari\'et\'e semi-ab\'elienne maximale d'un groupe alg\'ebrique connexe.

Si $G / k$ est un groupe r\'eductif, on note $G\ss$ son sous-groupe d\'eriv\'e et $G\sc$ le
rev\^etement simplement connexe de $G\ss$. On dispose donc d'un morphisme naturel $\rho : G\sc \to G$. On notera $Z_G$ le centre de $G$, $Z_{G\sc}$ celui de $G\sc$, et $T_G$ un $k$-tore maximal de $G$. On pose
$T_{G\sc} := \rho^{-1}(T_G)$, qui est un tore maximal de $G\sc$.

D\'esormais $G$ est un $k$-groupe alg\'ebrique connexe.

Par un th\'eor\`eme de Rosenlicht (voir \cite{Ros}, corollaire 3 du th\'eor\`eme 12), il existe $D \subset G$ un
$k$-sous-groupe distingu\'e connexe tel que $G\li := G / D$ soit
lin\'eaire et v\'erifie la propri\'et\'e universelle suivante : pour
tout $k$-groupe lin\'eaire $H$ et tout morphisme $f : G \to H$, il
existe un unique morphisme $f\li : G\li \to H$ tel que $f$ se factorise par le quotient $G \to G\li$.
Outre la suite exacte ainsi obtenue :
$$1 \to D \to G \xto{\pi} G\li \to 1 \, ,$$
on dispose du th\'eor\`eme de Chevalley : il existe un sous-groupe
caract\'eristique $G\lin \subset G$, qui est lin\'eaire connexe, et
tel que le quotient $G\abvar := G / G\lin$ soit une $k$-vari\'et\'e
ab\'elienne. 
$$1 \to G\lin \to G \xto{p} G\abvar \to 1 \, .$$
On a alors le fait suivant (voir \cite{Ros}, corollaire 1 du th\'eor\`eme 13 et corollaire 5 du th\'eor\`eme 16) :
\begin{lem}
Avec les notations pr\'ec\'edentes,
\begin{enumerate}
  \item $G = G\lin . D$.
    \item Le sous-groupe $D$ est central, i.e. $D \subset Z(G)$.
\end{enumerate}
\end{lem}
\label{lem Rosenlicht}
On d\'eduit de ce lemme le diagramme commutatif
exact de $k$-groupes suivant :
\begin{equation}
\label{diagramme groupes}
\xymatrix{
& 1 \ar[d] & 1 \ar[d] & & \\
1 \ar[r] & G\lin \cap D \ar[r] \ar[d] & D \ar[r]^{p'} \ar[d] & G\abvar \ar[r]
\ar[d]^{=} & 1 \\
1 \ar[r] & G\lin \ar[r] \ar[d]^{\pi'} & G \ar[d]^{\pi} \ar[r]^{p} &
G\abvar \ar[r] & 1 \\
& G\li \ar[r]^{=} \ar[d] & G\li \ar[d] & & \\
& 1 & 1 & & \, .
}
\end{equation}
Notons que le groupe $G\lin \cap D$ est un sous-groupe central de
$G\lin$, donc c'est un groupe lin\'eaire commutatif.

On fait d'abord l'hypoth\`ese suivante : \emph{le groupe $G\lin$ est
  r\'eductif}.
Alors $G\li$ est r\'eductif. Fixons $T_{G\li} \subset G\li$ un $k$-tore maximal. La suite exacte 
$$1 \to G\lin \cap D \to G\lin \xto{\pi'} G\li \to 1$$
assure que $T_G = T_{G\lin} := {\pi'}^{-1}(T_{G\li}) \subset G\lin$ est un tore maximal de $G\lin$.
D\'efinissons alors $\SA_G := \pi^{-1}(T_{G\li}) \subset G$. Le diagramme
(\ref{diagramme groupes}) induit un diagramme commutatif de groupes alg\'ebriques
commutatifs :
\begin{displaymath}
\xymatrix{
& 1 \ar[d] & 1 \ar[d] & & \\
1 \ar[r] & G\lin \cap D \ar[r] \ar[d] & D \ar[r]^{p'} \ar[d] & G\abvar \ar[r]
\ar[d]^{=} & 1 \\
1 \ar[r] & T_{G} \ar[r] \ar[d]^{\pi'} & \SA_G \ar[d]^{\pi} \ar[r]^{p} &
G\abvar \ar[r] & 1 \\
& T_{G\li} \ar[r]^{=} \ar[d] & T_{G\li} \ar[d] & & \\
& 1 & 1 & & \, .
}
\end{displaymath}

Le groupe $\SA_G$ est donc une extension de la vari\'et\'e ab\'elienne $G\abvar$ par le $k$-tore $T_{G}$. Donc $\SA_G$ est une vari\'et\'e semi-ab\'elienne sur $k$. C'est cette vari\'et\'e semi-ab\'elienne qui va jouer le r\^ole que jouait le tore maximal dans les groupes r\'eductifs.

\begin{lem}
Le groupe $\SA_G$ est une sous-vari\'et\'e semi-ab\'elienne maximale de $G$, i.e. toute vari\'et\'e semi-ab\'elienne contenue dans $G$ et contenant $\SA_G$ est \'egale \`a $\SA_G$.
\end{lem}

\begin{proof}
La preuve est \'evidente.
\end{proof}

Si $G\sc$ d\'esigne le rev\^etement simplement connexe du groupe d\'eriv\'e de $G\lin$, on a bien des morphismes naturels
$\rho_G : G\sc \to G$ et $\rho_G : T_{G\sc} \to \SA_G$
de sorte que le diagramme suivant soit commutatif :
\begin{equation}
\label{diag mod crois}
\xymatrix{
Z_{G\sc} \ar[r] \ar[d] & T_{G\sc} \ar[r] \ar[d] & G\sc \ar[d] \\
Z_G \ar[r] & \SA_G \ar[r] & G \, .
}
\end{equation}

\begin{lem}
\label{lem qis}
Le diagramme (\ref{diag mod crois}) d\'efinit des quasi-isomorphismes de modules crois\'es
$$[Z_{G\sc} \to Z_G] \to [T_{G\sc} \to \SA_G] \to [G\sc \to G] \, .$$
\end{lem}

\begin{proof}
On commence par v\'erifier que le morphisme $G\sc \to G$ d\'efinit un module crois\'e. Par le lemme \ref{lem Rosenlicht}, on sait que $Z_{G\lin} = Z_G \cap G\lin$ et que l'on a une suite exacte
$$0 \to Z_{G\lin} \to Z_G \to G\abvar \to 0 \, .$$
On d\'efinit alors un morphisme $G \to \textup{Aut}(G\sc)$ comme le compos\'e des morphismes suivants :
$$G \to G / Z_G \cong G\lin / Z_{G\lin} \cong G\sc / Z_{G\sc} = \textup{Int}(G\sc) \subset \textup{Aut}(G\sc) \, .$$
On v\'erifie alors imm\'ediatemment que cette action de $G$ sur $G\sc$ d\'efinit une structure de module crois\'e sur $G\sc \to G$. Avec cette d\'efinition, le diagramme (\ref{diag mod crois}) est bien form\'e de morphismes de modules crois\'es.

Pour montrer que les morphismes du lemme sont des quasi-isomorphismes, la preuve est identique \`a celle du lemme 2.4.1 de \cite{BorAMS}, en remarquant que $G = G\ss.Z_G = G\ss.\SA_G$ (voir le lemme \ref{lem Rosenlicht}) et que $\rho_G(T_{G\sc}) = \SA_G \cap G\ss$ et $\rho_G(Z_{G\sc}) = Z_G \cap G\ss$.
\end{proof}

Plus g\'en\'eralement, si $G\lin$ n'est pas suppos\'e r\'eductif, on note $G\uu$ le radical unipotent de $G\lin$, $\SA_G := \SA_{G/G\uu}$ et $C_G := [T_{G\sc} \to \SA_G]$.
%

\begin{exs}
\begin{itemize}
  \item Si $G$ est r\'eductif, on retrouve le complexe de tores $C_G = [T_{G\sc} \to T_G]$ utilis\'e par Borovoi (voir \cite{BorAMS}).
    \item Si $G$ est une vari\'et\'e semi-ab\'elienne, alors $C_G = [0 \to G]$ peut \^etre consid\'er\'e comme le $1$-motif associ\'e \`a $G$.
\end{itemize}
\end{exs}

Pour tout groupe connexe $G$, on a donc construit le complexe $C_G := [T_{G\sc} \xrightarrow{\rho_G} \SA_G]$.
On voit ce complexe comme un objet de la cat\'egorie des complexes de
groupes alg\'ebriques commutatifs sur $k$. L'objet qui nous int\'eresse est le complexe de modules galoisiens ``dual'' du complexe $C_G$, que l'on note $\widehat{C}_G$ :
$$\widehat{C}_G := [\widehat{T_G} \to (G\abvar)^*(\ov{k}) \oplus \widehat{T_{G\sc}}] \, ,$$
avec $\widehat{T_G}$ en degr\'e $0$. D\'efinissons aussi une variante de ce complexe :
$$\widehat{C}'_G := [\widehat{T_G} \to \Pic(\ov{G\abvar}) \oplus \widehat{T_{G\sc}}] \, .$$

Le complexe $\widehat{C}_G$ est li\'e au groupe $\Br(G)$. Le point
crucial est le th\'eor\`eme suivant, qui g\'en\'eralise le th\'eor\`eme 4.8 de
\cite{BvH} au cas des groupes non-lin\'eaires :
\begin{theo}
\label{theo qis}
Soit $G$ un groupe alg\'ebrique connexe. 
Il existe dans la cat\'egorie d\'eriv\'ee un isomorphisme canonique 
$$\widehat{C}'_G \cong \UPic(\ov{G})$$
et un triangle exact canonique
$$\widehat{C}_G \to \UPic(\ov{G}) \to \textup{NS}(\ov{G}\abvar)[-1] \to \widehat{C}_G[1] \, .$$
\end{theo}


\begin{proof}
L'isomorphisme $(\ov{G}\abvar)^*(\ov{k}) \cong \Pic^0(\ov{G}\abvar)$ induit une suite exacte
$$0 \to (\ov{G}\abvar)^*(\ov{k}) \to \Pic(\ov{G}\abvar) \to \textup{NS}(\ov{G}\abvar) \to 0 \, .$$
Par cons\'equent, il suffit de montrer que le complexe $[\widehat{T_G} \to \widehat{T_{G\sc}} \oplus \Pic(\ov{G}\abvar)]$ est quasi-isomorphe \`a $\UPic(\ov{G})$. Enfin, puisque le morphisme naturel $\UPic(\ov{G}/\ov{G\uu}) \to \UPic(\ov{G})$ est un isomorphisme, on peut supposer que $G\lin$ est r\'eductif. 

\begin{lem}
\label{lem qis toit}
Le diagramme naturel suivant
\begin{equation}
\label{gd diag UPic}
\xymatrix{
\widehat{T_G} \ar[r]^(.4){\widehat{\rho} \oplus \Delta'_{\textup{SA}_G / T_G}} & \widehat{T_{G\sc}} \oplus \Pic(\ov{G}\abvar) \\
\widehat{T_G} \oplus \ov{k}(G)^* / \ov{k}^* \ar[r]^{\lambda} \ar[u]^{\textup{pr}_1} \ar[d]^{\textup{pr}_2} & \UPic_{\ov{T}_G}(\ov{G})^1 \ar[u]^{f_1 \oplus f_2} \ar[d]^{\textup{pr}_{\textup{Div}}} \\
\ov{k}(G)^* / \ov{k}^* \ar[r]^{\textup{div}} & \textup{Div}(\ov{G})
}
\end{equation}
est commutatif, et il induit des quasi-isomorphismes entre les trois complexes horizontaux.
\end{lem}

\begin{proof}
Le groupe $\UPic_{\ov{T}_G}(\ov{G})^1$ est d\'efini dans \cite{BvHprep}, section 2, de la fa\c con suivante : si $H$ est un $k$-groupe lin\'eaire connexe agissant (\`a gauche) sur une $k$-vari\'et\'e g\'eom\'etriquement int\`egre $Y$, on note 
$$\UPic_{\ov{H}}(\ov{Y})^1 := \left\{ (D,z) \in \textup{Div}(\ov{Y})
 \times \ov{k}(H \times Y)^* : \left\{ \begin{array}{l} z_{h_1 h_2} (y) = z_{h_1}(y) . z_{h_2
 }(h_1^{-1}.y) \\
 \textup{div}(z) = m^* D - \textup{pr}_Y^* D
 \end{array} \right. \right\} \, ,$$
o\`u $m : H \times Y \to Y$ d\'esigne l'action de $H$ sur $Y$ et $\textup{pr}_Y : H \times Y \to Y$ d\'esigne la projection sur le second facteur. On dispose d'un complexe canonique (voir \cite{BvHprep}, d\'efinition 2.3)
$$\UPic_{\ov{H}}(\ov{Y}) := \left[\ov{k}(Y)^* / \ov{k}^* \xrightarrow{d} \UPic_{\ov{H}}(\ov{Y})^1\right]$$ 
d\'efini par $d(f) := \left( \textup{div}(f), \frac{m^* f}{\textup{pr}_Y^* f} \right)$, de sorte que $\textup{Ker}(d)$ soit canoniquement isomorphe \`a $U_{\ov{H}}(\ov{Y})$ et $\textup{Coker}(d)$ \`a $\Pic_{\ov{H}}(\ov{Y})$ (voir \cite{BvHprep}, section 2 et th\'eor\`eme 3.12).

D\'efinissons chacun des morphismes apparaissant dans le diagramme (\ref{gd diag UPic}).
Le morphisme $\lambda : \widehat{T_G} \oplus \ov{k}(G)^* / \ov{k}^* \to \UPic_{\ov{T}_G}(\ov{G})^1$ est d\'efini par la formule suivante : $\lambda(\chi,f) := (\chi.d^0(f) , \textup{div}(f))$,
o\`u $d^0(f) : (t,g) \mapsto f(t.g)/f(g)$. 
Le morphisme $f_1$ est la compos\'ee
$$f_1 : \UPic_{\ov{T}_G}(\ov{G})^1 \xrightarrow{\nu} \Pic_{\ov{T}_G}(\ov{G}) \to \Pic_{\ov{T}_{G\sc}}(\ov{G}\sc) \xleftarrow{\cong} \widehat{T_{G\sc}} \, ,$$
o\`u $\nu$ est le morphisme naturel (voir \cite{BvHprep}, th\'eor\`eme 3.12) et le dernier morphisme provient de \cite{KKV}, lemme 2.2 et proposition 2.3. De m\^eme, $f_2$ est la compos\'ee 
$$f_2 : \UPic_{\ov{T}_G}(\ov{G})^1 \xrightarrow{\nu} \Pic_{\ov{T}_G}(\ov{G}) \to \Pic_{\ov{T}_G}(\ov{\SA}_G) \xleftarrow{\cong} \Pic(\ov{G}\abvar) \, ,$$
o\`u le dernier morphisme provient de la proposition 5.1 de \cite{KKV}.
Les autres morphismes du diagramme (\ref{gd diag UPic}) sont les morphismes naturels.

Le diagramme (\ref{gd diag UPic}) est commutatif : pour le carr\'e inf\'erieur, c'est \'evident. Pour le carr\'e sup\'erieur, cela r\'esulte d'un calcul explicitant les diff\'erents morphismes.

Montrons que le diagramme (\ref{gd diag UPic}) d\'efinit des quasi-isomorphismes entre les complexes horizontaux. Consid\'erons d'abord le carr\'e inf\'erieur. Les noyaux de $\lambda$ et $\textup{div}$ sont clairement isomorphes via $\textup{pr}_2$ (ils sont canoniquement isomorphes \`a $U(\ov{G})$). Le morphisme $\textup{Coker}(\lambda) \to \textup{Coker}(\textup{div})$ s'identifie au morphisme naturel $\Pic_{\ov{T}_G}(\ov{G}) / H^1_{\textup{alg}}(\ov{T}_G,\mathcal{O}(\ov{G})^*) \to \Pic(\ov{G})$, et celui-ci est un isomorphisme par \cite{KKV}, lemme 2.2 (en utilisant $\Pic(\ov{T}_G) = 0$). Donc le carr\'e inf\'erieur d\'efinit un quasi-isomorphisme entre les complexes horizontaux.

Consid\'erons pour finir le carr\'e sup\'erieur : \`a nouveau, il est clair que le morphisme $\textup{pr}_1$ induit un isomorphisme $\textup{Ker}(\lambda) \cong \textup{Ker}(\widehat{\rho} \oplus \Delta'_{\textup{SA}_G / T_G})$. Montrons que $f_1 \oplus f_2 : \textup{Coker}(\lambda) \to \textup{Coker}(\widehat{\rho} \oplus \Delta'_{\textup{SA}_G / T_G})$ est un isomorphisme.
%
\begin{lem}
\label{lem Pic T G}
Le morphisme canonique $\Pic_{\ov{T}_G}(\ov{G}) \to \Pic_{\ov{T}_G}(\ov{G}\lin) \oplus \Pic(\ov{G}\abvar)$ est un isomorphisme.
\end{lem}

\begin{proof}
D\'efinissons les espaces homog\`enes $W := G / T_G$ et $W\lin := G\lin / T_G$. Consid\'erons le diagramme commutatif ``exact'' de $k$-vari\'et\'es suivant :
\begin{equation}
\label{diag W}
\xymatrix{
& 1 \ar[d] & 1 \ar[d] &   & \\
& T_G \ar[r]^= \ar[d] & T_G \ar[d] & & \\
1 \ar[r] & G\lin \ar[r] \ar[d] & G \ar[r] \ar[d] & G\abvar \ar[r] \ar[d]^= & 1 \\
1 \ar[r] & W\lin \ar[r]^{i} \ar[d] & W \ar[r]^{p} \ar[d] & G\abvar \ar[r] & 1 \\
& 1 & 1 & & \, .
}
\end{equation}
Remarquons que les inclusions $T_G \subset \SA_G \subset G$
induisent un morphisme naturel $s : G\abvar \to W$ de sorte que $p \circ s = \textup{id}_{G\abvar}$.
De la m\^eme fa\c con, on a un morphisme naturel $q : W \to W\lin$ obtenu en identifiant $W\lin$ \`a $G / \SA_G$, de sorte que $q \circ i = \textup{id}_{W\lin}$.
Consid\'erons alors la suite de morphismes suivante :
\begin{equation}
\label{sec W}
0 \to \Pic(\ov{G}\abvar) \xrightarrow{p^*} \Pic(\ov{W}) \xrightarrow{i^*} \Pic(\ov{W}\lin) \to 0 \, .
\end{equation}
Cette suite est clairement un complexe. L'existence des morphismes $s$ et $q$ assure son exactitude en $\Pic(\ov{G}\abvar)$ et en $\Pic(\ov{W}\lin)$, ainsi que l'existence d'une section du morphisme $i^*$.
Montrons que cette suite est exacte en $\Pic(\ov{W})$.

On dispose d'un diagramme commutatif, issu du diagramme (\ref{diag W}), dont les lignes et les colonnes sont exactes (hormis la troisi\`eme ligne en $\Pic(\ov{W})$) :
\begin{equation}
\xymatrix{
& 0 \ar[r] & \widehat{G} \ar[r]^a \ar@{^{(}->}[d]^{\alpha} & \widehat{G\lin} \ar@{^{(}->}[d]^{\alpha'} \ar[r]^t & \Pic(\ov{G}\abvar) \\
& & \widehat{T_G} \ar[r]^b_= \ar[d]^{\beta} & \widehat{T_G} \ar[d]^{\beta'} & \\
0 \ar[r] & \Pic(\ov{G}\abvar) \ar[r]^{p^*} \ar[d]^= & \Pic(\ov{W}) \ar[r]^{i^*} \ar@{->>}[d]^{\gamma} & \Pic(\ov{W}\lin) \ar[r] \ar@{->>}[d]^{\gamma'} & 0 \\
\widehat{G\lin} \ar[r]^t & \Pic(\ov{G}\abvar) \ar[r]^e & \Pic(\ov{G})
 \ar[r]^d & \Pic(\ov{G}\lin) & \, .
}
\end{equation}
Deux applications du lemme du serpent aux deux colonnes de ce diagramme assure l'existence d'une suite exacte canonique
$$0 \to \widehat{G\lin} / \widehat{G} \xrightarrow{h} \textup{Ker}(i^*) \xrightarrow{\gamma} \textup{Ker}(d) \to 0$$
qui s'ins\`ere dans le diagramme suivant \`a lignes exactes :
\begin{displaymath}
\xymatrix{
0 \ar[r] & \widehat{G\lin} / \widehat{G} \ar[r]^t \ar[d]^= & \Pic(\ov{G}\abvar) \ar[r]^{e} \ar[d]^{p^*} & \textup{Ker}(d) \ar[r] \ar[d]^= & 0 \\
0 \ar[r] &  \widehat{G\lin} / \widehat{G} \ar[r]^{h} & \textup{Ker}(i^*) \ar[r]^{\gamma} & \textup{Ker}(d) \ar[r] & 0 \, .
}
\end{displaymath}
On v\'erifie alors facilement que ce diagramme est commutatif 
et on conclut avec le lemme des cinq que $p^*$ identifie $\Pic(\ov{G}\abvar)$ \`a $\textup{Ker}(i^*)$, ce qui conclut la preuve de l'exactitude de la suite (\ref{sec W}), ce qui assure donc que l'on a un isomorphisme canonique
$$s^* \oplus i^* : \Pic(\ov{W}) \xrightarrow{\cong} \Pic(\ov{G}\abvar) \oplus \Pic(\ov{W}\lin) \, .$$
On conclut alors la preuve du lemme \ref{lem Pic T G} en utilisant les isomorphismes fonctoriels de \cite{KKV} (proposition 5.1) $\Pic(\ov{W}) \xrightarrow{\cong} \Pic_{\ov{T_G}}(\ov{G})$ et $\Pic(\ov{W}\lin) \xrightarrow{\cong} \Pic_{\ov{T_G}}(\ov{G}\lin)$.
\end{proof}

Utilisons maintenant le lemme \ref{lem Pic T G} pour montrer le lemme \ref{lem qis toit}. Gr\^ace \`a \cite{BvHprep}, th\'eor\`eme 3.12, le complexe central dans (\ref{gd diag UPic}) est canoniquement quasi-isomorphe au complexe
$$[\widehat{T_G} \xrightarrow{\lambda} \Pic_{\ov{T_G}}(\ov{G})] \, ,$$
puisque $U_{\ov{T}_G}(\ov{G}) = U_{\ov{T}_G}(\ov{G}\tor) = 0$, en utilisant les notations et le lemme 5.4 de \cite{BvHprep}.

Or ce complexe s'inscrit dans le diagramme commutatif suivant 
\begin{displaymath}
\xymatrix{
\widehat{T_G} \ar[rr]^(.3){\lambda\lin \oplus \Delta'_{\SA_G / T_G}} & & \Pic_{\ov{T}_G}(\ov{G}\lin) \oplus \Pic(\ov{G}\abvar) \\
\widehat{T_G}  \ar[rr]^{\lambda} \ar[u]^{=} & & \Pic_{\ov{T}_G}(\ov{G}) \ar[u]^{f_1 \oplus f_2} \, .
}
\end{displaymath}
qui est un isomorphisme entre les complexes horizontaux par le lemme \ref{lem Pic T G}. Enfin, on peut identifier canoniquement $\Pic_{\ov{T}_G}(\ov{G}\lin)$ \`a $\widehat{T_{G\sc}}$ (en utilisant la proposition 5.1 de \cite{KKV}), et donc le complexe $[\widehat{T_G} \xrightarrow{\lambda\lin \oplus \Delta'_{\SA_G / T_G}} \Pic_{\ov{T}_G}(\ov{G}\lin) \oplus \Pic(\ov{G}\abvar)]$ au complexe $[\widehat{T_G} \xrightarrow{\widehat{\rho} \oplus \Delta'_{\SA_G / T_G}} \widehat{T_{G\sc}} \oplus \Pic(\ov{G}\abvar)]$, ce qui assure que le carr\'e sup\'erieur dans le lemme \ref{lem qis toit} est un quasi-isomorphisme. 
\end{proof}

Il est alors clair que le lemme \ref{lem qis toit} implique le th\'eor\`eme \ref{theo qis}.
\end{proof}

\subsubsection{Cas des espaces homog\`enes de groupes lin\'eaires connexes}
Revenons maintenant au cas de l'espace homog\`ene $X$ de $G$. Le lien $L_{\ov{G}/X}$ sur $\ov{H}$, d\'efini par la donn\'ee de recollement induite par $\ov{x} \in X(\ov{k})$, d\'efinit une $k$-forme $Z_{H\red}$ du centre de $\ov{H}\red$. De m\^eme, ce lien d\'efinit une $k$-forme $Z_{H\sc}$ du centre de $\ov{H}\sc$. Fixons une d\'ecomposition de L\'evi de $\ov{H}$. L'inclusion induite par cette d\'ecomposition $Z_{H\red} \rightarrow G$ n'est pas en g\'en\'eral $\Gamma_k$-\'equivariante. En revanche, on v\'erifie qu'au niveau des modules de caract\`eres, le morphisme induit $\widehat{G} \to \widehat{Z_{H\red}}$ est $\Gamma_k$-\'equivariant et ne d\'epend pas de la d\'ecomposition de L\'evi choisie.
On dispose ainsi d'un complexe canonique de modules galoisiens de longueur $3$ :
$$\widehat{\ov{C}}_X := \left[\widehat{G} \to \widehat{Z_{\ov{H}\red}} \to \widehat{Z_{\ov{H}\sc}} \right] \, .$$

On peut comparer le complexe $C_{\ov{Z}/X}$ avec le complexe de modules galoisiens $\widehat{\ov{C}}_X$.

Pour cela, on consid\`ere le diagramme suivant :
\begin{equation}
\label{gd diag crucial}
\xymatrix{
\widehat{G} \ar[r] \ar[d]^= & \widehat{Z_{\ov{H}\red}} \ar[r] \ar[d]^{\cong} & \widehat{Z_{\ov{H}\sc}} \ar[d]^{\cong} \\
\widehat{G} \ar[r] & \Pic_{\ov{G}}(\ov{Z}) \ar[r] & \Pic(\ov{H'}) \\
\widehat{G} \oplus \ov{k}(Z)^* / \ov{k}^* \ar[r] \ar[u]^{\textup{pr}_1}
\ar[d]^{\textup{pr}_2} & \UPic_{\ov{G}}(\ov{Z})^1 \ar[r] \ar[u]^{\nu'} \ar[d]^{\mu} & \Pic(\ov{H'})
\ar[u]^{=} \ar[d]^{\cong} \\
\ov{k}(Z)^* / \ov{k}^* \ar[r] & \Div(\ov{Z}) \ar[r] &
\Pic'(\ov{Z}/\ov{X}) \, .
}
\end{equation}
D\'efinissons les deux morphismes $\widehat{Z_{\ov{H}\red}} \cong \Pic_{\ov{G}}(\ov{Z})$ et $\widehat{Z_{\ov{H}\sc}} \cong \Pic(\ov{H'})$. Pour le premier, c'est le morphisme de \cite{KKLV}, exemple 2.1 (voir aussi \cite{Pop}, th\'eor\`eme 4). Pour le second, on a un isomorphisme bien connu $\widehat{Z_{\ov{H}\sc}} \xrightarrow{\cong} \Pic(\ov{H}\sc / {Z_{\ov{H}\sc}})$, et le morphisme $\ov{H}\sc / {Z_{\ov{H}\sc}} \to \ov{H}\red / Z_{\ov{H}\red} = \ov{H}'$ est un isomorphisme de groupes alg\'ebriques, donc $\Pic(\ov{H'}) \xrightarrow{\cong} \Pic(\ov{H}\sc / {Z_{\ov{H}\sc}})$, d'o\`u le second isomorphisme recherch\'e.

On v\'erifie facilement que ce diagramme est commutatif. On voit ce diagramme comme un diagramme entre les complexes de modules galoisiens horizontaux. On v\'erifie que ce diagramme r\'ealise des quasi-isomorphismes entre les complexes des trois premi\`eres lignes.

On d\'eduit donc de ce diagramme un morphisme canonique dans la cat\'egorie d\'eriv\'ee : 
$$\widehat{\ov{C}}_X \to C_{\ov{G}/X} \, ,$$
qui s'int\`egre dans le triangle exact suivant :
$$\widehat{\ov{C}}_X \to C_{\ov{G}/X} \to \Pic(\ov{G})[-1] \to \widehat{\ov{C}}_X[1] \, .$$
On \'ecrit alors la suite exacte longue associ\'ee \`a ce triangle,
ainsi que la suite exacte du th\'eor\`eme \ref{theo torseurs}, et on
obtient les suites exactes suivantes :
$$0 \to \H^1(k, \widehat{\ov{C}}_X) \to \H^1(k, C_{\ov{G}/X}) \to
\Pic(\ov{G})^{\Gamma_k} \to \H^2(k, \widehat{\ov{C}}_X) \to \H^2(k,
C_{\ov{G}/X}) \to H^1(k, \Pic(\ov{G}))$$
$$0 \to \textup{Pic}(X) \to \H^1(k, C_{\ov{G}/X}) \to \textup{Br}(k) \to
\textup{Br}_1(X,G) \to \H^2(k, C_{\ov{G}/X}) \to N^3(k,\Gm) \, ,$$
o\`u $N^3(k, \Gm) := \textup{Ker}(H^3(k, \Gm) \to H^3(X, \Gm))$. On obtient aussi les isomorphismes $U(X) \xrightarrow{\cong} \H^0(k, C_{\ov{G}/X}) \xleftarrow{\cong} \H^0(k,\widehat{\ov{C}}_X)$.
On peut donc finalement \'enoncer le r\'esultat souhait\'e :
\begin{theo}
\label{theo general}
Soit $k$ un corps de caract\'eristique nulle, $G$ un $k$-groupe connexe,  $X$ un espace homog\`ene de $G$, \`a
stabilisateurs g\'eom\'etriques lin\'eaires connexes $\ov{H}$. 
D\'efinissons le complexe de modules galoisiens suivant
$$\widehat{\ov{C}}_X := \left[ \widehat{G} \to \widehat{Z_{\ov{H}\red}} \to \widehat{Z_{\ov{H}\sc}} \right] \, .$$
\begin{itemize}
  \item On a des isomorphismes naturels
$$\H^0(k,\widehat{\ov{C}}_X) \cong \H^0(k, C_{\ov{G}/X}) \cong U(X) \, .$$
\item On a des suites exactes :
\begin{equation}
\label{gd diag general}
0 \to \H^1(k, \widehat{\ov{C}}_X) \to \H^1(k, C_{\ov{G}/X}) \to
\Pic(\ov{G})^{\Gamma_k} \to \H^2(k, \widehat{\ov{C}}_X) \to \H^2(k,
C_{\ov{G}/X}) \to H^1(k, \Pic(\ov{G}))
\end{equation}
et
\begin{equation}
\label{gd diag general 2}
0 \to \textup{Pic}(X) \to \H^1(k, C_{\ov{G}/X}) \to \textup{Br}(k) \to \textup{Br}_1(X,G) \to \H^2(k,
 C_{\ov{G}/X}) \to N^3(k, \Gm) \, ,
\end{equation}
o\`u $N^3(k, \Gm) := \textup{Ker}(H^3(k, \Gm) \to H^3(X, \Gm))$.
En particulier, si $\Pic(\ov{G}) = 0$, on a une suite exacte naturelle :
$$0 \to \Pic(X) \to \H^1(k, \widehat{\ov{C}}_X) \to \Br(k) \to
      \Br_1(X,G) \to \H^2(k, \widehat{\ov{C}}_X) \to N^3(k, \Gm) \, .$$
\end{itemize}
\end{theo}

%
%

On en d\'eduit facilement le corollaire suivant concernant le groupe de Picard de $X$ :
\begin{cor}
\label{cor general 1}
Consid\'erons les inclusions naturelles 
$$\H^1(k, \widehat{\ov{C}}_X) \xrightarrow{a} \H^1(k, C_{\ov{G}/X}) \xleftarrow{b} \textup{Pic}(X)$$
induites par (\ref{gd diag general}) et (\ref{gd diag general 2}).
\begin{itemize}
  \item Si $X(k) \neq \emptyset$ ou plus g\'en\'eralement si
	$\textup{Br}(k)$ s'injecte dans $\Br(X)$, alors $b$
    est un isomorphisme, donc on a une suite exacte canonique
$$0 \to  \H^1(k, \widehat{\ov{C}}_X) \xrightarrow{a} \textup{Pic}(X) \to \Pic(\ov{G})^{\Gamma_k} \, .$$
\item Si $G$ est lin\'eaire et $G\ss$ simplement connexe (i.e. $\Pic(\ov{G}) = 0$), alors $a$ est un isomorphisme et on a
  donc une suite exacte canonique
$$0 \to \textup{Pic}(X) \xrightarrow{b} \H^1(k, \widehat{\ov{C}}_X) \to
      \textup{Ker}(\Br(k) \to \Br(X)) \, .$$
\item Si $\textup{Br}(k)$ s'injecte dans $\Br(X)$ (par exemple $X(k)
      \neq \emptyset$) et $\Pic(\ov{G}) = 0$, alors on a un isomorphisme canonique
$$\H^1(k, \widehat{\ov{C}}_X) \cong \textup{Pic}(X)\, .$$
\end{itemize}
\end{cor}

On d\'eduit de m\^eme le corollaire suivant, concernant le groupe de Brauer de $X$ :
\begin{cor}
\label{cor general 2}
Consid\'erons les morphismes naturels
$$\H^2(k, \widehat{\ov{C}}_X) \xrightarrow{a} \H^2(k, C_{\ov{G}/X}) \xleftarrow{b} \textup{Br}_a(X,G)$$
induits par (\ref{gd diag general}) et (\ref{gd diag general 2}).
\begin{itemize}
  \item Si $X(k) \neq \emptyset$ ou si plus g\'en\'eralement $H^3(k,
	\Gm)$ s'injecte dans $H^3(X, \Gm)$, alors $b$
    est un isomorphisme, donc on a une suite exacte canonique
$$\Pic(\ov{G})^{\Gamma_k} \to \H^2(k, \widehat{\ov{C}}_X) \xrightarrow{a} \textup{Br}_a(X, G) \to H^1(k, \Pic(\ov{G})) \, .$$
\item Si $\Pic(\ov{G}) = 0$, alors $a$ est un isomorphisme et on a
  donc une suite exacte canonique
$$0 \to \textup{Br}_a(X,G) \xrightarrow{b} \H^2(k, \widehat{\ov{C}}_X)
      \to \textup{Ker}(H^3(k, \Gm) \to H^3(X, \Gm)) \, .$$
\item Si $H^3(k, \Gm)$ s'injecte dans $H^3(X, \Gm)$ et $\Pic(\ov{G}) = 0$, on a un isomorphisme canonique
$$\H^2(k, \widehat{\ov{C}}_X) \cong \textup{Br}_a(X,G) \, .$$
\end{itemize}
\end{cor}

\subsubsection{Cas des espaces homog\`enes munis d'un point rationnel}
Dans cette sous-section, on s'int\'eresse au cas d'un espace homog\`ene $X$ d'un groupe connexe $G$ quelconque, sans l'hypoth\`ese $\Pic(\ov{G}) = 0$, mais en supposant que $X(k) \neq \emptyset$. On a besoin d'enlever l'hypoth\`ese $\Pic(\ov{G}) = 0$ pour traiter les espaces homog\`enes de groupes non lin\'eaires.

On se donne donc un $k$-sous-groupe connexe $H \subset G$ tel que $X \cong G/H$.
Suivant \cite{BCTS}, proposition 3.1, quitte \`a quotienter $G$ et $H$ par le sous-groupe $Z_G \cap H \subset H$, on peut se ramener au cas o\`u le stabilisateur est lin\'eaire connexe. Aussi, dans les \'enonc\'es de cette section, l'hypoth\`ese de lin\'earit\'e sur $H$ n'est-elle pas restrictive. Dans la suite, on suppose donc $H$ lin\'eaire.

La choix d'une d\'ecomposition de L\'evi du groupe $H$ induit un
morphisme de complexes 
$$j : C_H \to C_G \, ,$$
i.e. un carr\'e commutatif de vari\'et\'es semi-ab\'eliennes :
\begin{displaymath}
\xymatrix{
T_{H\sc} \ar[r] \ar[d] & T_{G\sc} \ar[d] \\
T_H \ar[r] & \SA_G \, .
}
\end{displaymath}
Ce morphisme de complexes est d\'efini \`a isomorphisme pr\`es (deux d\'ecompositions de L\'evi d\'efinissent des carr\'es commutatifs isomorphes).
On d\'efinit alors $C_X$ comme le c\^one du morphisme de complexes
$C_H \to C_G$, \`a savoir comme le complexe de longueur $3$ suivant :
$$C_X := \textup{C\^one}(C_H \to C_G) = [T_{H\sc} \to T_H \oplus
T_{G\sc} \to \SA_G]$$
form\'e de vari\'et\'es semi-ab\'eliennes.
On peut alors d\'efinir un nouveau complexe ${\widehat{C}''}_X$ comme le dual de
ce complexe dans la cat\'egorie des complexes de $1$-motifs,
c'est-\`a-dire comme le complexe de $1$-motifs suivant :
$${\widehat{C}''}_X := [\SA_G^* \to \widehat{T_H} \oplus \widehat{T_{G\sc}} \to
\widehat{T_{H\sc}}]$$
(si $S$ est une vari\'et\'e semi-ab\'elienne, $S^*$ est le $1$-motif dual
du $1$-motif $[0 \to S]$ associ\'e \`a $S$).
En \'ecrivant chacun des $1$-motifs comme un complexe de longueur $2$,
${\widehat{C}''}_X$ s'identifie au complexe de complexes suivant :
\begin{displaymath}
\xymatrix{
\widehat{T_{G\lin}} \ar[d] \ar[r] & \widehat{T_H} \oplus
\widehat{T_{G\sc}} \ar[r] \ar[d] & \widehat{T_{H\sc}} \ar[d] \\
(G\abvar)^* \ar[r] & 0 \ar[r] & 0 \, .
}
\end{displaymath}
On s'int\'eresse plus exactement aux sections sur $\ov{k}$ du c\^one
de ce double complexe, \`a savoir au complexe de modules galoisiens :
$$\widehat{C}_X := [\widehat{T_{G\lin}} \rightarrow (G\abvar)^*(\ov{k}) \oplus \widehat{T_H} \oplus
\widehat{T_{G\sc}} \rightarrow \widehat{T_{H\sc}}] \, .
$$
On introduit \'egalement une variante de ce complexe not\'ee $\widehat{C}'_X$, qui est utile dans la suite :
$$\widehat{C}'_X := [\widehat{T_{G\lin}} \rightarrow \Pic(\ov{G}\abvar) \oplus \widehat{T_H} \oplus
\widehat{T_{G\sc}} \rightarrow \widehat{T_{H\sc}}] \, .$$
On dispose en particulier d'un triangle exact canonique
$$\widehat{C}_X \to \widehat{C}'_X \to \textup{NS}(\ov{G}\abvar)[-1] \to \widehat{C}_X[1] \, .$$

\begin{exs}
\begin{itemize}
\item Si $G$ est lin\'eaire, on trouve le complexe de modules
  galoisiens de type fini :
$$\widehat{C}_X = [\widehat{T_G} \rightarrow \widehat{T_H} \oplus
\widehat{T_{G\sc}} \rightarrow \widehat{T_{H\sc}}] \, .$$
\item Si $G\ss$ est simplement connexe, on a un quasi-isomorphisme naturel 
$$\widehat{C}_X = [\widehat{G\tor} \rightarrow (G\abvar)^*(\ov{k}) \oplus
\widehat{T_H} \rightarrow \widehat{T_{H\sc}}] \, .$$
\item Si $G$ est lin\'eaire et $G\ss$ est simplement connexe, on
  obtient le complexe de modules galoisiens de type fini suivant : $\widehat{C}_X = [\widehat{G} \rightarrow \widehat{T_H} \rightarrow \widehat{T_{H\sc}}]$.
\item Si $H = 1$, i.e. $X = G$, on obtient le complexe  $\widehat{C}_G = [\widehat{T_{G\lin}} \rightarrow (G\abvar)^*(\ov{k}) \oplus \widehat{T_{G\sc}}]$ qui, si $G$ est lin\'eaire, est exactement $[\widehat{T_G} \to \widehat{T_{G\sc}}] \cong \UPic(\ov{G})$ (voir \cite{BvH}, th\'eor\`eme 4.8).
\item Si $G$ est semi-simple simplement connexe, on obtient $\widehat{C}_X := [0 \rightarrow \widehat{T_H} \rightarrow \widehat{T_{H\sc}}]$, qui s'identifie \`a $\UPic(\ov{H})[-1]$.
\end{itemize}
\end{exs}

Dans tous les cas, on remarque que les complexes $\widehat{C}_X$ et $\widehat{C}'_X$ s'int\`egrent naturellement
dans les triangles exacts suivants :
$$\widehat{C}_X \to \widehat{C}_G \to \widehat{C}_H \to
\widehat{C}_X[1]$$
$$\widehat{C}'_X \to \widehat{C}'_G \to \widehat{C}_H \to
\widehat{C}'_X[1] \, .$$
Remarquons que ce dernier triangle exact s'identifie, gr\^ace au th\'eor\`eme \ref{theo qis}, au triangle exact
$$\widehat{C}'_X \to \UPic(\ov{G}) \to \UPic(\ov{H}) \to \widehat{C}'_X[1] \, .$$
On voit donc que $\widehat{C}'_X[1]$ est un c\^one du morphisme $\UPic(\ov{G}) \to \UPic(\ov{H})$.

On dispose \'egalement d'une pr\'esentation du complexe
$\widehat{C}_X$ \`a l'aide du centre de $H\red$ :
en effet, le centre $Z_{H\red}$ de $H\red$ est contenu dans le tore maximal $T_H$ de $H\red$, et le diagramme commutatif suivant
\begin{displaymath}
\xymatrix{
Z_{H\sc} \ar[r] \ar[d] & T_{H\sc} \ar[d] \\
Z_{H\red} \ar[r] & T_H
}
\end{displaymath}
d\'efinit un quasi-isomorphisme de complexes entre $[Z_{H\sc} \to Z_{H\red}]$
et $[T_{H\sc} \to T_H]$, lequel induit par dualit\'e un
quasi-isomorphisme canonique
$$\widehat{C}_X \xrightarrow{\textup{qis}} [\widehat{T_{G\lin}} \rightarrow (G\abvar)^*(\ov{k}) \oplus \widehat{Z_{H\red}} \oplus
\widehat{T_{G\sc}} \rightarrow \widehat{Z_{H\sc}}] \,
.$$

C'est ce complexe faisant intervenir le centre de $H$ qui est utile. On souhaite relier le complexe $\widehat{C}_X$ au groupe de Brauer de $X$. Le point crucial
est le th\'eor\`eme suivant :
\begin{theo}
\label{theo qis X}
Soit $X = G/H$, avec $G$ connexe et $H$ lin\'eaire connexe.
Il existe alors dans la cat\'egorie d\'eriv\'ee un isomorphisme canonique 
$$\widehat{C}'_X \cong C_{\ov{G}/X}$$
et un triangle exact canonique
$$\widehat{C}_X \to C_{\ov{G}/X} \to \textup{NS}(\ov{G}\abvar)[-1] \to \widehat{C}_X[1] \, .$$
\end{theo}

\begin{proof}
On fixe une d\'ecomposition de L\'evi pour $H$.
Posons $Z' := T_G / Z_{H\red}$, $Z := G / Z_{H\red}$ et $H' := H\red / Z_{H\red}$.
On construit un diagramme commutatif analogue au diagramme (\ref{gd diag crucial}), \`a savoir :
\begin{equation}
  \label{gd diag 2}
\xymatrix{
\widehat{T_G} \ar[r] \ar[d]^= & \Pic(\ov{G}\abvar) \oplus \widehat{T_{G\sc}} \oplus \widehat{Z_{\ov{H}\red}} \ar[r] \ar[d]^{\cong} & \widehat{Z_{\ov{H}\sc}} \ar[d]^{\cong} \\
\widehat{T_G} \ar[r] & \Pic(\ov{G}\abvar) \oplus \Pic_{\ov{T_G}}(\ov{G}\lin) \oplus \Pic_{\ov{T_G}}(\ov{Z'}) \ar[r] & \Pic(\ov{H'}) \\
\widehat{T_G} \ar[r] \ar[u]^= & \Pic_{\ov{T_G}}(\ov{G}) \oplus \Pic_{\ov{T_G}}(\ov{Z'}) \ar[r] \ar[u]^{\cong} & \Pic(\ov{H'}) \ar[u]^= \\
\widehat{T_G} \ar[r] \ar[u]^= & \Pic_{\ov{T_G}}(\ov{Z}) \ar[r] \ar[u] & \Pic(\ov{H'}) \ar[u]^{=} \\
\widehat{T_G} \oplus \ov{k}(Z)^* / \ov{k}^* \ar[r] \ar[u]^{\textup{pr}_1} \ar[d]^{\textup{pr}_2} & \UPic_{\ov{T_G}}(\ov{Z})^1 \ar[r] \ar[u]^{\nu'} \ar[d]^{\mu} & \Pic(\ov{H'}) \ar[u]^{=} \ar[d]^{\cong} \\
\ov{k}(Z)^* / \ov{k}^* \ar[r] & \Div(\ov{Z}) \ar[r] & \Pic'(\ov{Z}/\ov{X}) \, .
}
\end{equation}
Le morphisme entre la premi\`ere et la deuxi\`eme ligne est clairement un isomorphisme de complexes. Le morphisme entre la troisi\`eme et la deuxi\`eme est un isomorphisme de complexes par le lemme \ref{lem Pic T G}. Montrons que le morphisme entre les quatri\`eme et troisi\`eme lignes est un isomorphisme. Il suffit de montrer que le morphisme central 
$$\Pic_{\ov{T_G}}(\ov{Z}) \to \Pic_{\ov{T_G}}(\ov{G}) \oplus \Pic_{\ov{T_G}}(\ov{Z'})$$
est un isomorphisme.
\begin{lem}
\label{lem Pic T Z}
Soient $H_1 \subset H_2 \subset G$ trois $k$-groupes. On suppose que $G$ est connexe et que $H_2$ est lin\'eaire connexe. Alors on a des isomorphismes canoniques
$$\Pic_{H_2}(G/H_1) \xrightarrow{\cong} \Pic_{H_2}(G) \oplus \Pic_{H_2}(H_2/H_1) \xleftarrow{\cong} \Pic(G/H_2) \oplus \Pic_{H_2}(H_2/H_1) \, .$$
\end{lem}

\begin{proof}
Notons $K := \textup{Ker}(\widehat{H_1} \to \Pic(G/H_1))$. On consid\`ere le diagramme commutatif exact suivant (voir notamment \cite{KKV}, proposition 5.1) : 
\begin{displaymath}
\xymatrix{
0 \ar[d] & 0 \ar[d] & 0 \ar[d] & & \widehat{H_1} / U(G) \ar@{^{(}->}[d] &  \\
U(G/H_2) \ar@{^{(}->}[r] \ar[d]^= & U(G/H_1) \ar[r] \ar[d] & \widehat{H_2} \ar[r] \ar[d]^= & \Pic_{H_2}(G/H_1) \ar[r] \ar[d] & \Pic(G/H_1) \ar[r] \ar[d] & \Pic(H_2) \ar[d]^= \\
U(G/H_2) \ar@{^{(}->}[r] \ar[d] & U(G) \ar[r] \ar@{->>}[d] & \widehat{H_2} \ar[r] \ar[d] & \Pic_{H_2}(G) \ar[r] & \Pic(G) \ar[r] & \Pic(H_2) \\
0 & K & 0 & & & \, .
}
\end{displaymath}
Le lemme du serpent assure alors que l'on a une suite exacte canonique
$$0 \to K \to \textup{Ker}(\Pic_{H_2}(G/H_1) \to  \Pic_{H_2}(G)) \to \widehat{H_1} / U(G) \to 0$$
qui s'int\`egre dans le diagramme exact
\begin{displaymath}
\xymatrix{
0 \ar[r] & K \ar[r] \ar[d]^= & \widehat{H_1} \ar[r] \ar[d] & \widehat{H_1} / U(G) \ar[r] \ar[d]^= & 0 \\
0 \ar[r] & K \ar[r] & \textup{Ker}(\Pic_{H_2}(G/H_1) \to  \Pic_{H_2}(G)) \ar[r]  & \widehat{H_1} / U(G) \ar[r] & 0 \, ,
}
\end{displaymath}
o\`u le morphisme central est la compos\'ee $\widehat{H_1} \xrightarrow{\cong} \Pic_{H_2}(H_2/H_1) \to \Pic_{H_2}(G/H_1)$.
On v\'erifie facilement que le diagramme pr\'ec\'edent est commutatif, et le lemme des 5 assure donc que l'on a des isomorphismes canoniques 
$$\widehat{H_1} \xrightarrow{\cong} \Pic_{H_2}(H_2/H_1) \xrightarrow{\cong} \textup{Ker}(\Pic_{H_2}(G/H_1) \to  \Pic_{H_2}(G)) \, .$$
Par cons\'equent, la suite naturelle
$$0 \to \Pic_{H_2}(H_2/H_1) \to \Pic_{H_2}(G/H_1) \to  \Pic_{H_2}(G)$$
est exacte. Enfin, le morphisme $\Pic_{H_2}(G/H_1) \to  \Pic_{H_2}(G)$ est scind\'e via la compos\'ee
$$\Pic_{H_2}(G) \xleftarrow{\cong} \Pic(G/H_2) \hookrightarrow \Pic_{H_2}(G/H_1) \, ,$$
d'o\`u le lemme.
\end{proof}

On applique alors ce lemme \`a $H_1 = Z_{H\red}$ et $H_2 = T_G$ pour obtenir l'isomorphisme souhait\'e, et donc les complexes des lignes 3 et 4 du diagramme (\ref{gd diag 2}) sont isomorphes.

Poursuivons la preuve du th\'eor\`eme : il est clair que le morphisme entre les lignes 5 et 4 r\'ealise un quasi-isomorphisme entre les complexes correspondants.

Montrons enfin que le morphisme entre les lignes 5 et 6 est un quasi-isomorphisme : en degr\'e $0$, c'est imm\'ediat. En degr\'e 1, c'est clair en utilisant que $s : \Pic(\ov{H}') \to \Pic'(\ov{Z}/\ov{X})$ est injectif et en s'inspirant de la preuve du lemme \ref{lem qis toit}. Pour le degr\'e $2$, cela r\'esulte du fait que le morphisme canonique $s : \Pic(\ov{H}') \to \Pic'(\ov{Z}/\ov{X})$ est un isomorphisme gr\^ace au corollaire \ref{cor Leray 2}.

Finalement, le diagramme (\ref{gd diag 2}) d\'efinit un isomorphisme canonique dans la cat\'egorie d\'eriv\'ee 
$$\widehat{C}'_X \xrightarrow{\cong} C_{\ov{G}/X} \, ,$$
ce qui conclut la preuve du th\'eor\`eme.
\end{proof}

D\'eduisons des th\'eor\`emes \ref{theo torseurs} et \ref{theo qis X}, le r\'esultat suivant :
\begin{theo}
\label{theo principal}
Soit $X = G / H$, avec $G$ connexe, $H$ lin\'eaire connexe. Alors on a des isomorphismes, fonctoriels en
$(X,Y,H,\pi)$ :
$$\H^0(k, \widehat{C}_X) \xrightarrow{\cong} \H^0(k, \widehat{C}'_X) \xrightarrow{\cong} U(X) \, ,$$
$$\H^1(k,\widehat{C}'_X) \xrightarrow{\cong} \Pic(X) \, , \, \, \, \H^2(k, \widehat{C}'_X) \xrightarrow{\cong} \Br_a(X,G) \, ,$$
et une suite exacte fonctorielle :
$$0 \to \H^1(k,\widehat{C}_X) \to \Pic(X) \to \textup{NS}(\ov{G}\abvar)^{\Gamma_k} \to \H^2(k, \widehat{C}_X) \to \Br_a(X,G) \to H^1(k, \textup{NS}(\ov{G}\abvar)) \, .$$
En particulier, si le groupe $G$ est lin\'eaire, on a trois isomorphismes
$$\H^0(k, \widehat{C}_X) \xrightarrow{\cong} U(X) \, , \, \, \, \H^1(k,\widehat{C}_X) \xrightarrow{\cong} \Pic(X) \, , \, \, \, \H^2(k, \widehat{C}_X) \xrightarrow{\cong} \Br_a(X,G) \, .$$
\end{theo}


\begin{proof}
Par le th\'eor\`eme \ref{theo qis}, on dispose d'un isomorphisme $\widehat{C}'_X \cong C_{\ov{G}/X}$
et d'un triangle exact
$$\widehat{C}_X \to C_{\ov{G}/X} \to \textup{NS}(\ov{G}\abvar)[-1] \to \widehat{C}_X[1] \, .$$
On consid\`ere la suite exacte longue d'hypercohomologie associ\'ee \`a ce triangle, et le corollaire \ref{coro torseurs pt ratio} assure alors le r\'esultat, puisque $G(k) \neq \emptyset$.
\end{proof}

\subsection{Comparaison avec des r\'esultats ant\'erieurs}
On compare ici les r\'esultats obtenus aux th\'eor\`emes \ref{theo general}, \ref{theo qis X} et \ref{theo principal}, ainsi qu'aux corollaires \ref{cor general 1} et \ref{cor
  general 2}, avec des r\'esultats ant\'erieurs de Sansuc,
Borovoi et Van Hamel, Harari et Szamuely.

On poursuit avec les m\^emes notations que pr\'ec\'edemment : $X$ est
un espace homog\`ene d'un groupe connexe $G$ sur un corps $k$ de
caract\'eristique nulle. On suppose les stabilisateurs
g\'eom\'etriques $\ov{H}$ de $X$ connexes.

\subsubsection{Cas o\`u $\ov{H} = 1$}
On s'int\'eresse donc aux torseurs sous un groupe connexe $G$. 
Suivant le lemme 5.2.(iii) de \cite{BvH}, on dispose d'un quasi-isomorphisme canonique $\varphi : \UPic(\ov{X}) \xrightarrow{\cong} \UPic(\ov{G})$.

On voit que le th\'eor\`eme \ref{theo general} et ses corollaires g\'en\'eralisent
les r\'esultats de Sansuc (voir \cite{San}, section 6) qui donnent des formules
pour $U(X)$, $\Pic(X)$ et $\textup{Br}_a(X)$ quand $G$ est un $k$-tore ou
un $k$-groupe semi-simple.


Plus r\'ecemment, citons les travaux de Harari et Szamuely (voir
\cite{HSz2}, section 4) qui obtiennent une formule pour $\textup{Br}_a(X)$ dans
le cas o\`u $G$ est une vari\'et\'e semi-ab\'elienne. Plus exactement, sous l'hypoth\`ese $H^3(k, \Gm) = 0$, les auteurs construisent un morphisme canonique 
$$\H^1(k, G^*) \to \Br_a(X) \, ,$$
qui est compatible en un certain sens avec l'accouplement de Brauer-Manin.
Sous la m\^eme hypoth\`ese $H^3(k, \Gm) = 0$, on \'etend et on pr\'ecise ici cette
construction pour un espace principal homog\`ene $X$ sous un groupe connexe $G$ quelconque avec le
th\'eor\`eme \ref{theo qis} qui fournit une suite exacte
$$\textup{NS}(\ov{G}^{\textup{ab}})^{\Gamma_k} \to \H^2(k, \widehat{C}_G) \xrightarrow{a} \textup{Br}_a(X) \to H^1(k, \textup{NS}(\ov{G}^{\textup{ab}})) \, .$$
En effet, par construction, si $G$ est une vari\'et\'e semi-ab\'elienne, on a un quasi-isomorphisme naturel $G^*(\ov{k}) \cong \widehat{C}_G [1]$ (avec les conventions de \cite{HSz2} pour le $1$-motif $G^*$).

Enfin, Borovoi et van Hamel obtiennent dans \cite{BvH} des formules dans le cas o\`u $G$ est lin\'eaire connexe quelconque. Plus pr\'ecis\'emment, ils montrent par exemple que l'on a un isomorphisme canonique 
$$\Br_a(X) \cong \H^2(k, [\widehat{T}_G \to \widehat{T}_G\sc]) \, ,$$
sous l'hypoth\`ese $H^3(k, \Gm) = 0$. On retrouve ici ce r\'esultat comme cas particulier du th\'eor\`eme \ref{theo general}.

\subsubsection{Cas o\`u $G$ est lin\'eaire}
Dans ce cas, le r\'esultat principal est d\^u \`a Borovoi et van Hamel
(voir \cite{BvH2} corollaire 3.2 et \cite{BvHprep} th\'eor\`eme 7.2) : il d\'ecrit le groupe de Brauer alg\'ebrique de $X$. Plus pr\'ecis\'emment, les auteurs montrent (sans hypoth\`ese de connexit\'e pour les stabilisateurs) l'existence d'un morphisme injectif canonique
$$\Br_a(X) \to \H^2(k, [\widehat{\ov{G}} \to \widehat{\ov{H}}])$$
qui est un isomorphisme si $X(k) \neq \emptyset$ ou $H^3(k, \Gm) = 0$.
On peut retrouver ce r\'esultat, si $\ov{H}$ est connexe, \`a  partir du corollaire \ref{cor general 2}, en calculant les \'el\'ements alg\'ebriques dans $\H^2(k, \widehat{\ov{C}}_X)$.


\begin{rem}
Notons que tous ces r\'esultats ant\'erieurs se limitent \`a d\'ecrire
le groupe de Brauer \emph{alg\'ebrique}
$\textup{Br}_a(X)$ de l'espace homog\`ene $X$. Dans le pr\'esent
texte, on consid\`ere le groupe $\textup{Br}_a(X,G)$, contenant $\Br_a(X)$, mais qui peut contenir
des \'el\'ements transcendants.
\end{rem}

\subsection{Lien avec le complexe $\textup{UPic}(\overline{X})$ :
  groupe de Brauer alg\'ebrique}

On relie ici les r\'esultats obtenus aux sections pr\'ec\'edentes \`a
ceux de Borovoi et van Hamel dans \cite{BvH}, \cite{BvH2} et \cite{BvHprep}, qui traitent du groupe
de Brauer alg\'ebrique.

Dans cette section, $H$ est un groupe alg\'ebrique lin\'eaire connexe, $X$ est une $k$-vari\'et\'e lisse g\'eom\'etriquement int\'egre et $\pi : Y \xrightarrow{H} X$ est un torseur sous $H$.

On dispose comme plus haut d'un diagramme commutatif de torseurs :
\begin{displaymath}
\xymatrix{
Y \ar[r]^{H\uu} \ar[ddr]_{H} & Z' \ar[dd]^{H\red} \ar[rd]^{Z_{H\red}} & \\
& & Z \ar[ld]^{H'} \\
& X & \, .
}
\end{displaymath}

On v\'erifie alors facilement que le diagramme commutatif
\begin{displaymath}
\xymatrix{
\ov{k}(X)^*/\ov{k}^* \ar[r] \ar[d]^{p^*} & \textup{Div}(\ov{X}) \ar[r] \ar[d]^{p^*} & 0 \ar[d] \\
\ov{k}(Z)^*/\ov{k}^* \ar[r] \ar[d] & \textup{Div}(\ov{Z}) \ar[r] \ar[d] & \Pic(\ov{H'}) \ar[d] \\
0 \ar[r] & 0 \ar[r] & \Pic(\ov{H'})/\Pic(\ov{Z})
}
\end{displaymath}
induit un triangle exact 
$$\UPic(\ov{X}) \to C_{\ov{Y}/X} \to \Pic(\ov{H'})/\Pic(\ov{Z})[-2] \to \UPic(\ov{X})[1]\, .$$
Enfin, consid\'erons le diagramme commutatif \`a lignes exactes suivant :
\begin{equation}
\label{diagramme 2 BD}
\xymatrix{
0 \ar[r] & \Pic(\ov{X}) \ar[r] \ar[d]^= & \Pic(\ov{Z}) \ar[r]^{\varphi'} \ar[d] & \Pic(\ov{H'}) \ar[r]^{\Delta_{Z/X}} \ar[d] & \Br(\ov{X}) \ar[r] \ar[d]^= & \Br(\ov{Z}) \ar[d] \\
& \Pic(\ov{X}) \ar[r] & \Pic(\ov{Y}) \ar[r]^{\varphi} & \Pic(\ov{H}) \ar[r]^{\Delta_{Y/X}} & \Br(\ov{X}) \ar[r] & \Br(\ov{Y}) \, .
}
\end{equation}
Puisque $\Pic(\ov{H'}) \to \Pic(\ov{H})$ est surjectif (cf lemme \ref{lem surj Pic}), une chasse au diagramme assure que (\ref{diagramme 2 BD}) induit un isomorphisme
$$\Pic(\ov{H'})/\Pic(\ov{Z}) \cong \Pic(\ov{H})/\Pic(\ov{Y}) \cong \textup{Ker}(\Br(\ov{X}) \to \Br(\ov{Y}))\, .$$
Finalement, on en d\'eduit la proposition suivante :
\begin{prop}
\label{prop UPic torseurs}
Soit $H$ un $k$-groupe lin\'eaire connexe, $X$ une $k$-vari\'et\'e lisse
 g\'eom\'etriquement int\`egre et $\pi : Y \xrightarrow{H} X$ un torseur sous $H$. On a alors un triangle exact canonique 
$$\UPic(\ov{X}) \to C_{\ov{Y}/X} \to \left( \Pic(\ov{H}) / \Pic(\ov{Y}) \right) [-2] \to \UPic(\ov{X})[1]$$
o\`u $\left( \Pic(\ov{H}) / \Pic(\ov{Y}) \right) \cong \textup{Ker}(\Br(\ov{X}) \to \Br(\ov{Y}))$.
\end{prop}

En particulier, la suite exacte longue associ\'ee \`a ce triangle exact induit la suite exacte :
$$0 \to \H^2(k, \UPic(\ov{X})) \to \H^2(k, C_{\ov{Y}/X}) \to \Br(\ov{X})^{\Gamma_k} \, .$$
Alors le th\'eor\`eme \ref{theo torseurs} et la proposition 2.19 de \cite{BvH} assurent que cette suite exacte s'int\`egre dans le diagramme commutatif \`a lignes et colonnes exactes suivant :
\begin{displaymath}
\xymatrix{
0 \ar[r] & \Br_a(X) \ar[r] \ar@{^{(}->}[d] & \Br_a(X,Y) \ar[r] \ar@{^{(}->}[d] & \Br(\ov{X})^{\Gamma_k} \ar[d]^= \\
0 \ar[r] & \H^2(k, \UPic(\ov{X})) \ar[r] \ar[d] & \H^2(k, C_{\ov{Y}/X}) \ar[r] \ar[d] & \Br(\ov{X})^{\Gamma_k} \ar[d] \\
0 \ar[r] & N^3(k, \Gm) \ar[r]^= & N^3(k, \Gm) \ar[r] & 0
}
\end{displaymath}
qui illustre le lien entre le th\'eor\`eme \ref{theo torseurs}
(deuxi\`eme colonne) et le r\'esultat de Borovoi et van Hamel
(premi\`ere colonne).

Consid\'erons \'egalement le cas des espaces homog\`enes : soit d'abord $G$ un $k$-groupe alg\'ebrique \emph{lin\'eaire} connexe, tel que $\Pic(\ov{G}) = 0$, et $X$ un espace homog\`ene de $G$ \`a stabilisateurs g\'eom\'etriques $\ov{H}$ connexes. On a introduit plus haut le complexe 
$$\widehat{\ov{C}}_X := [\widehat{\ov{G}} \to \widehat{Z_{\ov{H}\red}} \to \widehat{Z_{\ov{H}\sc}}]$$ 
et on peut \'egalement le comparer au complexe $\UPic(\ov{X})$. En effet, par le th\'eor\`eme 3.1 de \cite{BvH2} (th\'eor\`eme 5.8 de \cite{BvHprep}), on a un isomorphisme naturel dans la cat\'egorie d\'eriv\'ee :
$$\UPic(\ov{X}) \cong [\widehat{\ov{G}} \to \widehat{\ov{H}}] \, ,$$
dont on d\'eduit ais\'ement la proposition suivante :
\begin{prop}
\label{prop UPic esp hom}
Soit $G$ un $k$-groupe alg\'ebrique lin\'eaire connexe, tel que $G\ss$ soit simplement connexe (i.e. $\Pic(\ov{G}) = 0$), et $X$ un espace homog\`ene de $G$ \`a stabilisateurs g\'eom\'etriques $\ov{H}$ connexes.
On a un triangle exact canonique dans la cat\'egorie d\'eriv\'ee
$$\UPic(\ov{X}) \to \widehat{\ov{C}}_X \to \Pic(\ov{H})[-2] \to \UPic(\ov{X})[1] \, .$$
\end{prop}

Consid\'erons pour finir le cas d'un espace homog\`ene $X$ d'un groupe $G$ connexe non n\'ecessairement lin\'eaire, \`a stabilisateurs connexes.
Si $X(k) \neq \emptyset$, alors on a un isomorphisme $X \cong G/H$. Comme plus haut, on peut supposer $H$ lin\'eaire connexe et $G\ss$ simplement connexe.
Le r\'esultat suivant \'etend le r\'esultat de Borovoi et van Hamel
(th\'eor\`eme 3.1 de \cite{BvH2} et th\'eor\`eme 5.8 de \cite{BvHprep}) au cas non lin\'eaire :
\begin{prop}
Soit $X = G/H$ un espace homog\`ene d'un groupe connexe $G$, avec $G\ss$ simplement connexe et $H$ lin\'eaire (pas forc\'ement connexe).
On dispose alors d'un isomorphisme canonique dans la cat\'egorie d\'eriv\'ee :
$$\UPic(\ov{X}) \cong [\widehat{G\lin} \to \Pic(\ov{G}\abvar) \oplus \widehat{H}] \, .$$
\end{prop}

\begin{proof}
Notons $X\lin := G\lin / H$.
Alors le diagramme commutatif suivant :
\begin{displaymath}
\xymatrix{
\widehat{G\lin} \ar[r] \ar[d]^= & \Pic(\ov{G}\abvar) \oplus \widehat{H} \ar[d]^{\cong} \\
\widehat{G\lin} \ar[r] & \Pic_{\ov{G}\lin}(\ov{G}) \oplus \Pic_{\ov{G}\lin}(\ov{X}\lin) \\
\widehat{G\lin} \ar[r] \ar[u]^= & \Pic_{\ov{G}\lin}(\ov{X}) \ar[u]^{\cong} \\
\widehat{G\lin} \oplus \ov{k}(X)^* / \ov{k}^* \ar[r] \ar[u] \ar[d] & \UPic_{\ov{G}\lin}(\ov{X})^1 \ar[u] \ar[d] \\
\ov{k}(X)^* / \ov{k}^* \ar[r] & \textup{Div}(\ov{X})
}
\end{displaymath}
r\'ealise un isomorphisme dans la cat\'egorie d\'eriv\'ee (voir lemme \ref{lem Pic T Z})
$$[\widehat{G\lin} \to \Pic(\ov{G}\abvar) \oplus \widehat{H}] \cong \UPic(\ov{X}) \, .$$
\end{proof}

Par cons\'equent, le th\'eor\`eme \ref{theo qis} assure que le triangle exact de la proposition \ref{prop UPic torseurs}
$$\UPic(\ov{X}) \to C_{\ov{G}/X} \to \Pic(\ov{H})[-2] \to \UPic(\ov{X})[1]$$
est repr\'esentable par le triangle exact \'evident de complexes
$$[\widehat{G\lin} \to \Pic(\ov{G}\abvar) \oplus \widehat{H}] \to [\widehat{G\lin} \to \Pic(\ov{G}\abvar) \oplus \widehat{Z_{\ov{H}\red}} \to \widehat{Z_{\ov{H}\sc}}] \to $$
$$\to \widehat{\mu_H}(\ov{k})[-2] \to [\widehat{G\lin} \to \Pic(\ov{G}\abvar) \oplus \widehat{H}][1] \, .$$


\bigskip

\begin{thebibliography}{KKLV89}
\expandafter\ifx\csname fonteauteurs\endcsname\relax
\def\fonteauteurs{\scshape}\fi

\bibitem[BCTS08]{BCTS}
Mikhail \bgroup\fonteauteurs\bgroup Borovoi\egroup\egroup{}, Jean-Louis
  \bgroup\fonteauteurs\bgroup Colliot-Th{\'e}l{\`e}ne\egroup\egroup{} et
  Alexei~N. \bgroup\fonteauteurs\bgroup Skorobogatov\egroup\egroup{} :
\newblock The elementary obstruction and homogeneous spaces.
\newblock {\em Duke Math. J.}, 141(2)\string:\penalty500\relax 321--364, 2008.

\bibitem[BD10]{BD}
Mikhail \bgroup\fonteauteurs\bgroup Borovoi\egroup\egroup{} et Cyril
  \bgroup\fonteauteurs\bgroup Demarche\egroup\egroup{} :
\newblock Manin obstruction to strong approximation for homogeneous spaces.
\newblock {\em Comment. Math. Helv.}, 2010.
\newblock \`A para\^itre.

\bibitem[BLR90]{BLR}
Siegfried \bgroup\fonteauteurs\bgroup Bosch\egroup\egroup{}, Werner
  \bgroup\fonteauteurs\bgroup L{\"u}tkebohmert\egroup\egroup{} et Michel
  \bgroup\fonteauteurs\bgroup Raynaud\egroup\egroup{} :
\newblock {\em N\'eron models}, volume~21 de {\em Ergebnisse der Mathematik und
  ihrer Grenzgebiete (3) [Results in Mathematics and Related Areas (3)]}.
\newblock Springer-Verlag, Berlin, 1990.

\bibitem[Bor93]{BorDuke}
Mikhail \bgroup\fonteauteurs\bgroup Borovoi\egroup\egroup{} :
\newblock Abelianization of the second nonabelian {G}alois cohomology.
\newblock {\em Duke Math. J.}, 72(1)\string:\penalty500\relax 217--239, 1993.

\bibitem[Bor98]{BorAMS}
Mikhail \bgroup\fonteauteurs\bgroup Borovoi\egroup\egroup{} :
\newblock Abelian {G}alois cohomology of reductive groups.
\newblock {\em Mem. Amer. Math. Soc.}, 132(626)\string:\penalty500\relax
  viii+50, 1998.

\bibitem[BvH06]{BvH2}
Mikhail \bgroup\fonteauteurs\bgroup Borovoi\egroup\egroup{} et Joost van
  \bgroup\fonteauteurs\bgroup Hamel\egroup\egroup{} :
\newblock Extended {P}icard complexes for algebraic groups and homogeneous
  spaces.
\newblock {\em C. R. Math. Acad. Sci. Paris}, 342(9)\string:\penalty500\relax
  671--674, 2006.

\bibitem[BvH09]{BvH}
Mikhail \bgroup\fonteauteurs\bgroup Borovoi\egroup\egroup{} et Joost van
  \bgroup\fonteauteurs\bgroup Hamel\egroup\egroup{} :
\newblock Extended {P}icard complexes and linear algebraic groups.
\newblock {\em J. Reine Angew. Math.}, 627\string:\penalty500\relax 53--82,
  2009.

\bibitem[BvH10]{BvHprep}
Mikhail \bgroup\fonteauteurs\bgroup Borovoi\egroup\egroup{} et Joost van
  \bgroup\fonteauteurs\bgroup Hamel\egroup\egroup{} :
\newblock Extended equivariant {P}icard complexes and homogeneous spaces.
\newblock Pr\'epublication, disponible sur
  \url{http://arxiv.org/abs/1010.3414}, 2010.

\bibitem[CE56]{CE}
Henri \bgroup\fonteauteurs\bgroup Cartan\egroup\egroup{} et Samuel
  \bgroup\fonteauteurs\bgroup Eilenberg\egroup\egroup{} :
\newblock {\em Homological algebra}.
\newblock Princeton University Press, Princeton, N. J., 1956.

\bibitem[Che60]{Che}
Claude \bgroup\fonteauteurs\bgroup Chevalley\egroup\egroup{} :
\newblock Une d\'emonstration d'un th\'eor\`eme sur les groupes alg\'ebriques.
\newblock {\em J. Math. Pures Appl. (9)}, 39\string:\penalty500\relax 307--317,
  1960.

\bibitem[Con02]{Con}
Brian \bgroup\fonteauteurs\bgroup Conrad\egroup\egroup{} :
\newblock A modern proof of {C}hevalley's theorem on algebraic groups.
\newblock {\em J. Ramanujan Math. Soc.}, 17(1)\string:\penalty500\relax 1--18,
  2002.

\bibitem[CTK06]{CTK}
Jean-Louis \bgroup\fonteauteurs\bgroup Colliot-Th{\'e}l{\`e}ne\egroup\egroup{}
  et Boris~{\`E}. \bgroup\fonteauteurs\bgroup Kunyavski{\u\i}\egroup\egroup{} :
\newblock Groupe de {P}icard et groupe de {B}rauer des compactifications lisses
  d'espaces homog\`enes.
\newblock {\em J. Algebraic Geom.}, 15(4)\string:\penalty500\relax 733--752,
  2006.

\bibitem[CTS87]{CTS}
Jean-Louis \bgroup\fonteauteurs\bgroup Colliot-Th{\'e}l{\`e}ne\egroup\egroup{}
  et Jean-Jacques \bgroup\fonteauteurs\bgroup Sansuc\egroup\egroup{} :
\newblock La descente sur les vari\'et\'es rationnelles. {II}.
\newblock {\em Duke Math. J.}, 54(2)\string:\penalty500\relax 375--492, 1987.

\bibitem[CTX09]{CTX}
Jean-Louis \bgroup\fonteauteurs\bgroup Colliot-Th{\'e}l{\`e}ne\egroup\egroup{}
  et Fei \bgroup\fonteauteurs\bgroup Xu\egroup\egroup{} :
\newblock Brauer-{M}anin obstruction for integral points of homogeneous spaces
  and representation by integral quadratic forms.
\newblock {\em Compos. Math.}, 145(2)\string:\penalty500\relax 309--363, 2009.
\newblock Avec un appendice par Dasheng Wei et Fei Xu.

\bibitem[Dou06]{Dou}
Jean-Claude \bgroup\fonteauteurs\bgroup Douai\egroup\egroup{} :
\newblock Sur la 2-cohomologie galoisienne de la composante r\'esiduellement
  neutre des groupes r\'eductifs connexes d\'efinis sur les corps locaux.
\newblock {\em C. R. Math. Acad. Sci. Paris}, 342(11)\string:\penalty500\relax
  813--818, 2006.

\bibitem[FI73]{FI}
Robert \bgroup\fonteauteurs\bgroup Fossum\egroup\egroup{} et Birger
  \bgroup\fonteauteurs\bgroup Iversen\egroup\egroup{} :
\newblock On {P}icard groups of algebraic fibre spaces.
\newblock {\em J. Pure Appl. Algebra}, 3\string:\penalty500\relax 269--280,
  1973.

\bibitem[Gir71]{Gir}
Jean \bgroup\fonteauteurs\bgroup Giraud\egroup\egroup{} :
\newblock {\em Cohomologie non ab\'elienne}.
\newblock Springer-Verlag, Berlin, 1971.
\newblock Die Grundlehren der mathematischen Wissenschaften, Band 179.

\bibitem[Gro68]{GBr}
Alexander \bgroup\fonteauteurs\bgroup Grothendieck\egroup\egroup{} :
\newblock Le groupe de {B}rauer. {I}, {II} et {III}.
\newblock \emph{In} {\em Dix {E}xpos\'es sur la {C}ohomologie des {S}ch\'emas},
  pages 46--188. North-Holland, Amsterdam, 1968.

\bibitem[HS02]{HSk}
David \bgroup\fonteauteurs\bgroup Harari\egroup\egroup{} et Alexei~N.
  \bgroup\fonteauteurs\bgroup Skorobogatov\egroup\egroup{} :
\newblock Non-abelian cohomology and rational points.
\newblock {\em Compositio Math.}, 130(3)\string:\penalty500\relax 241--273,
  2002.

\bibitem[HS08]{HSz2}
David \bgroup\fonteauteurs\bgroup Harari\egroup\egroup{} et Tam{\'a}s
  \bgroup\fonteauteurs\bgroup Szamuely\egroup\egroup{} :
\newblock Local-global principles for 1-motives.
\newblock {\em Duke Math. J.}, 143(3)\string:\penalty500\relax 531--557, 2008.

\bibitem[KKLV89]{KKLV}
Friedrich \bgroup\fonteauteurs\bgroup Knop\egroup\egroup{}, Hanspeter
  \bgroup\fonteauteurs\bgroup Kraft\egroup\egroup{}, Domingo
  \bgroup\fonteauteurs\bgroup Luna\egroup\egroup{} et Thierry
  \bgroup\fonteauteurs\bgroup Vust\egroup\egroup{} :
\newblock Local properties of algebraic group actions.
\newblock \emph{In} {\em Algebraische {T}ransformationsgruppen und
  {I}nvariantentheorie}, volume~13 de {\em DMV Sem.}, pages 63--75.
  Birkh\"auser, Basel, 1989.

\bibitem[KKV89]{KKV}
Friedrich \bgroup\fonteauteurs\bgroup Knop\egroup\egroup{}, Hanspeter
  \bgroup\fonteauteurs\bgroup Kraft\egroup\egroup{} et Thierry
  \bgroup\fonteauteurs\bgroup Vust\egroup\egroup{} :
\newblock The {P}icard group of a {$G$}-variety.
\newblock {\em Algebraische Transformationsgruppen und Invariantentheorie.
  DMV-Seminar}, 13\string:\penalty500\relax 77--88, 1989.

\bibitem[Kot84]{Kot}
Robert~E. \bgroup\fonteauteurs\bgroup Kottwitz\egroup\egroup{} :
\newblock Stable trace formula: cuspidal tempered terms.
\newblock {\em Duke Math. J.}, 51(3)\string:\penalty500\relax 611--650, 1984.

\bibitem[Mil80]{Mil}
James~S. \bgroup\fonteauteurs\bgroup Milne\egroup\egroup{} :
\newblock {\em \'{E}tale cohomology}, volume~33 de {\em Princeton Mathematical
  Series}.
\newblock Princeton University Press, Princeton, N.J., 1980.

\bibitem[Pop74]{Pop}
Vladimir~L. \bgroup\fonteauteurs\bgroup Popov\egroup\egroup{} :
\newblock Picard groups of homogeneous spaces of linear algebraic groups and
  one-dimensional homogeneous vector fiberings.
\newblock {\em Izv. Akad. Nauk SSSR Ser. Mat.}, 38\string:\penalty500\relax
  294--322, 1974.

\bibitem[Ray70]{Ray}
Michel \bgroup\fonteauteurs\bgroup Raynaud\egroup\egroup{} :
\newblock {\em Faisceaux amples sur les sch\'emas en groupes et les espaces
  homog\`enes}.
\newblock Lecture Notes in Mathematics, Vol. 119. Springer-Verlag, Berlin,
  1970.

\bibitem[Ros56]{Ros}
Maxwell \bgroup\fonteauteurs\bgroup Rosenlicht\egroup\egroup{} :
\newblock Some basic theorems on algebraic groups.
\newblock {\em Amer. J. Math.}, 78\string:\penalty500\relax 401--443, 1956.

\bibitem[San81]{San}
Jean-Jacques \bgroup\fonteauteurs\bgroup Sansuc\egroup\egroup{} :
\newblock Groupe de {B}rauer et arithm\'etique des groupes alg\'ebriques
  lin\'eaires sur un corps de nombres.
\newblock {\em J. reine angew. Math.}, 327\string:\penalty500\relax 12--80,
  1981.

\bibitem[Wei94]{Wei}
Charles~A. \bgroup\fonteauteurs\bgroup Weibel\egroup\egroup{} :
\newblock {\em An introduction to homological algebra}, volume~38 de {\em
  Cambridge Studies in Advanced Mathematics}.
\newblock Cambridge University Press, Cambridge, 1994.

\end{thebibliography}

\end{document}